\documentclass[11pt]{amsart} 
\usepackage{amssymb}
\usepackage{amsmath}
\usepackage{setspace}
\usepackage{amsthm}
\usepackage{amsfonts}
\usepackage{amstext}
\usepackage{mathtext}
\usepackage[T2A,T1]{fontenc}
\usepackage[cp1251]{inputenc}%
\usepackage[english,russian]{babel}%

\newtheorem{theorem}{рЕНПЕЛЮ}

\newtheorem{proposition}[theorem]{оПЕДКНФЕМХЕ}

\newtheorem{remark}[theorem]{гЮЛЕВЮМХЕ}

\newtheorem{definition}[theorem]{нОПЕДЕКЕМХЕ}

\begin{document}

\title{On some additive problems in number theory (v.2)}

\author{Andrei Allakhverdov}
\address{46 Sedykh str, flat 5, 220103, Minsk, Belarus}
\email{vandall@tut.by}

\begin{abstract}
In this article, we consider some problems in additive number theory.
Elementary solutions for the Goldbach-Euler conjecture and the twin primes conjecture are presented. The method employed also makes it possible to obtain some interesting results related to the densities of sequences and gaps between primes. A stronger version of the Goldbach-Euler conjecture, in which the strengthening is established by reducing the set of all primes to the set of twin prime pairs, is also presented.
The proof is based on the direct construction of the double sieve Eratosthenes-type and does not use empirical or heuristic methods.
\end{abstract}




\date{date}

\thanks{}

\maketitle


\hfill \begin{minipage}[h]{0.75\textwidth}
All results of the profoundest mathematical investigation must ultimately be expressible in the simple form of properties of integers.

 \begin{flushright}
Leopold Kronecker
\end{flushright}
 
Regarding the relative powers of elementary sieve methods and the analytical methods one usually considers that the latter should be more powerful... 
But history has shown that such views are not totally correct.

 \begin{flushright}
H.E. Richert, Lectures on Sieve Methods, \\
Tata Institute of Fundamental Research, 1976
\end{flushright}
 \end{minipage}

\section{нАНЯМНБЮМХЪ Х ЯНЦКЮЬЕМХЪ} 
дНЙЮГЮРЕКЭЯРБЮ НЯМНБЮМШ МЮ ЯБНИЯРБЮУ ЙКЮЯЯНБ БШВЕРНБ $\mathbf{Z}/6\mathbf{Z}$, Ю ХЛЕММН: 
ЙКЮЯЯШ БШВЕРНБ $\bar{1}_6$ Х $\bar{5}_6$ ЯНДЕПФЮР БЯЕ МЕВЕРМШЕ ОПНЯРШЕ ВХЯКЮ, ГЮ ХЯЙКЧВЕМХЕЛ ВХЯКЮ $3$; 
КЧАНИ ХГ ВЕРМШУ ЙКЮЯЯНБ БШВЕРНБ ОН ЛНДСКЧ $6$ ЛНФЕР АШРЭ ОПЕДЯРЮБКЕМ ЙЮЙ ЯСЛЛЮ ХКХ ПЮГМНЯРЭ ЙКЮЯЯНБ $\bar{1}_6$ Х/ХКХ $\bar{5}_6$;
ОНЯКЕДНБЮРЕКЭМНЯРХ, ЯНЯРЮБКЕММШЕ ХГ ЛМНФЕЯРБ $\bar{1}_6$ Х $\bar{5}_6$, СОНПЪДНВЕММШУ ОН БНГПЮЯРЮМХЧ ХКХ САШБЮМХЧ, УНПНЬН ЯРПСЙРСПХПНБЮМШ ОН ОПНЯРШЛ ВХЯКЮЛ. 
щРН ОНГБНКЪЕР ОНЯРПНХРЭ ДБНИМНЕ ПЕЬЕРН щПЮРНЯТЕМНБЮ РХОЮ.

оСЯРЭ ${P}$ -- ЛМНФЕЯРБН БЯЕУ ОПНЯРШУ ВХЯЕК. дЮКЕЕ ЛШ БЕГДЕ ОНКЮЦЮЕЛ $p \in \mathcal{P}$, ЦДЕ $\mathcal{P} = {P} \setminus \{ 2, 3\}$.

бШПЮФЕМХЕ $\# S_m$ НАНГМЮВЮЕР ВХЯКН МЕМСКЕБШУ ВКЕМНБ ОНЯКЕДНБЮРЕКЭМНЯРХ $S_m$.

\section{оПЕДБЮПХРЕКЭМШЕ ГЮЛЕВЮМХЪ}

\subsection{уНПНЬН ЯРПСЙРСПХПНБЮММШЕ ОНЯКЕДНБЮРЕКЭМНЯРХ}\label{primes_and_composite}

\begin{definition} 
{оНЯКЕДНБЮРЕКЭМНЯРЭ $S$ УНПНЬН ЯРПСЙРСПХПНБЮМЮ ОН ВХЯКС $q$, ЕЯКХ ХМДЕЙЯШ ВКЕМНБ ОНЯКЕДНБЮРЕКЭМНЯРХ $S$, ЙПЮРМШУ ВХЯКС $q$, НАПЮГСЧР ЮПХТЛЕРХВЕЯЙСЧ ОПНЦПЕЯЯХЧ Я ПЮГМНЯРЭЧ $q$.} 
\end{definition}

оСЯРЭ  $ A = \{ \, a_i \colon a_i = 6i-1, i \in \mathbf{N} \} $ Х ${ B} = \{ \, b_i \colon b_i = 6i+1, i \in \mathbf{N}  \} $ --  ДБЕ ОНЯКЕДНБЮРЕКЭМНЯРХ МЕВЕРМШУ ВХЯЕК:  
\tabcolsep=0.15em
 \begin{center}
\begin{tabular} {rrrrrrrrrrrrrrrrrrrrrr}
{${A}$  = } \{ &{5,} &{11,} &{17,} &{23,} &{29,} &{35,} &{41,} &{47,} &{53,} &{59,} &{65,} &{71,} &{77,} &{83,} &{89,} &{95,} &{101,} & \ldots \}, \\
{${B}$  = } \{ &{7,} &{13,} &{19,} &{25,} &{31,} &{37,} &{43,} &{49,} &{55,} &{61,} &{67,} &{73,} &{79,} &{85,} &{91,} &{97,} &{103,} & \ldots \}.     
\end{tabular}
\end{center} 
\begin{theorem}  
оНЯКЕДНБЮРЕКЭМНЯРХ $A$ Х $B$ УНПНЬН ЯРПСЙРСПХПНБЮМШ ОН БЯЕЛ ОПНЯРШЛ ВХЯКЮЛ $p \in \mathcal{P}$.
\end{theorem}   
\begin{proof}
EЯКХ $a_i$ ЯНЯРЮБМНЕ, РН МЮИДСРЯЪ РЮЙХЕ $j \not=0$, $ k \not= 0$, ВРН АСДЕР ХЛЕРЭ ЛЕЯРН ПЮБЕМЯРБН
\begin{align*} 
 { a_i=6i - 1 = a_j  b_k = (6j - 1)(6k + 1) = 36jk - 6k + 6j - 1}. 
\end{align*}
б ЩРНЛ ЯКСВЮЕ ЛШ ХЛЕЕЛ ДБЮ БШПЮФЕМХЪ ДКЪ $i$:
\begin{align*}
    i &= k\,(6j - 1) + j = k a_j +j, 	\\ 
    i &= j\,(6k + 1) - k = jb_k -k,        
  \end{align*}
ЙНРНПШЕ НОПЕДЕКЪЧР ДБЮ ЯЕЛЕИЯРБЮ ЮПХТЛЕРХВЕЯЙХУ ОПНЦПЕЯЯХИ:
\begin{align} 
	a_j \mid a_i & \Leftrightarrow i = + j + k a_j, \label{a_a_composite} \\ 
	b_k \mid a_i & \Leftrightarrow i = - k + j b_k, \label{a_b_composite}
	\end{align}	
еЯКХ $b_i$ ЯНЯРЮБМНЕ, РНЦДЮ МЮИДСРЯЪ РЮЙХЕ $j\not=0$, $ j'\not=0,$ ХКХ $ k\not=0, k'\not=0$ , ВРН АСДЕР ХЛЕРЭ ЛЕЯРН ОН ЙПЮИМЕИ ЛЕПЕ НДМН ХГ ПЮБЕМЯРБ
\begin{align*}
  b_i &= 6i + 1 = a_j a_{j'} = \left( {6j - 1} \right)\left( {6j'- 1} \right)  = 6\left( {6 j j' - j - j' } \right) + 1, \\  
  b_i &= 6i + 1 = b_k b_{k'} = \left( {6k + 1} \right)\left( {6k'+ 1} \right)  = 6\left( {6 k k' + k + k'} \right) + 1.   
\end{align*} 
б ЩРНЛ ЯКСВЮЕ РЮЙФЕ ХЛЕЕЛ ДБЮ БШПЮФЕМХЪ ДКЪ $i$:
\begin{align*}
   i &= j' ( 6j - 1 ) - j = j' a_j - j, 	\\   
   i &= k' ( 6k + 1 ) + k = k'b_k + k,            	
\end{align*}
ЙНРНПШЕ НОПЕДЕКЪЧР ДБЮ ЯЕЛЕИЯРБЮ ЮПХТЛЕРХВЕЯЙХУ ОПНЦПЕЯЯХИ: 	
	\begin{align}  
	a_j \mid b_i &\Leftrightarrow i = - j + j' a_j ,  \label{b_a_composite}	\\
	b_k \mid b_i &\Leftrightarrow i = + k + k' b_k.   \label{b_b_composite}	          
	\end{align}
хГ БШПЮФЕМХИ (\ref{a_a_composite}), (\ref{a_b_composite}), (\ref{b_a_composite}) Х (\ref{b_b_composite}) ЯКЕДСЕР, ВРН ЙЮФДЮЪ ХГ ОНЯКЕДНБЮРЕКЭМНЯРЕИ $A$ Х $B$ УНПНЬН ЯРПСЙРСПХПНБЮМЮ ЙЮЙ ОН БЯЕЛ ВХЯКЮЛ $a_j \in A$, РЮЙ Х ОН БЯЕЛ ВХЯКЮЛ $b_k \in B$. 
нЙНМВЮРЕКЭМН, СРБЕПФДЕМХЕ РЕНПЕЛШ ЯКЕДСЕР ХГ РНЦН, ВРН $A \cup B \supset  \mathcal P$.
\end{proof}

\subsection{оПНЖЕДСПЮ ОПНЯЕХБЮМХЪ} \label{sieve_method}рЮЙХЛ НАПЮГНЛ, ДКЪ КЧАНЦН $p \in \mathcal{P}$ Б ЙЮФДНИ ХГ ОНЯКЕДНБЮРЕКЭМНЯРЕИ ${A}$ Х ${B}$ ЯСЫЕЯРБСЕР ПНБМН НДМЮ ОНДОНЯКЕДНБЮРЕКЭМНЯРЭ, ВКЕМШ ЙНРНПНИ ЙПЮРМШ ОПНЯРНЛС $p$ Х ЙНРНПШЕ НАПЮГСЧР ЮПХТЛЕРХВЕЯЙСЧ ОПНЦПЕЯЯХЧ Я ПЮГМНЯРЭЧ  $6p$.   
хМДЕЙЯШ ВКЕМНБ ЩРНИ ОНДОНЯКЕДНБЮРЕКЭМНЯРХ НАПЮГСЧР ЮПХТЛЕРХВЕЯЙСЧ ОПНЦПЕЯЯХЧ Я ПЮГМНЯРЭЧ $p$.  

оПНЖЕДСПЮ ОПНЯЕХБЮМХЪ ОНЯКЕДНБЮРЕКЭМНЯРХ $S$ ОН ВХЯКС $p$ ЯНЯРНХР Б ГЮЛЕМЕ ЯНЯРЮБМШУ ВХЯЕК ЩРНИ ОНЯКЕДНБЮРЕКЭМНЯРХ, ЙПЮРМШУ ВХЯКС $p$ МСКЪЛХ.  
оНЯКЕДНБЮРЕКЭМНЯРЭ $S$, ОПНЯЕЪММСЧ ОН ВХЯКС $p$ ЛШ НАНГМЮВЮЕЛ ЯХЛБНКНЛ $ S {\setminus \lambda p}$.  
оНЯКЕДНБЮРЕКЭМНЯРЭ $S$, ОПНЯЕЪММСЧ ОН БЯЕЛ $p \in \mathcal P$ ЛШ НАНГМЮВЮЕЛ ЯХЛБНКНЛ $S {\setminus \lambda \mathcal P}$.

\subsection{оНЯКЕДНБЮРЕКЭМНЯРХ $\mathcal A$, $\mathcal B$, $\mathcal L$ Х $\mathcal R$} \label{sequences L and R}

оСЯРЭ $\mathcal A = A \setminus \lambda \mathcal P$ Х ОСЯРЭ $\mathcal B = B \setminus \lambda \mathcal P$.
нОПЕДЕКХЛ ОНЯКЕДНБЮРЕКЭМНЯРХ $\mathcal L$ Х $\mathcal R$  
ЯКЕДСЧЫХЛ НАПЮГНЛ: 
$ {\mathcal L} = \left\{ l_i \colon {l_i} = i \text{ ЕЯКХ } {a_i} \in \mathcal{P}; \right. \left. {l_i} = 0 \text{ ЕЯКХ } {a_i} \notin \mathcal{P} \right\}$ Х
$ {\mathcal R} = \left\{ r_i \colon {r_i} = i \text{ ЕЯКХ } {b_i} \in \mathcal{P}; {r_i} = 0  \text{  ЕЯКХ  } {b_i} \notin \mathcal{P} \right\},$
РН ЕЯРЭ  
\tabcolsep=0.15em
 \begin{center}
\begin{tabular} {rrrrrrrrrrrrrrrrrrrrrr}
{${\mathcal A}$  = } \{ &{5,} &{11,} &{17,} &{23,} &{29,} &{0,} &{41,} &{47,} &{53,} &{59,} &{0,}  &{71,} &{0,}  &{83,} &{89,} &{0,}  &{101,} &  \ldots \}, \\
\vspace{10pt}
{${\mathcal L}$  = } \{ &{1,} &{2,} &{3,} &{4,} &{5,} &{0,} &{7,} &{8,} &{9,} &{10,} &{0,}  &{12,} &{0,}  &{14,} &{15,} &{0,}  &{17,} &  \ldots \}, \\
{${\mathcal B}$  = } \{ &{7,} &{13,} &{19,} &{0,} &{31,} &{37,} &{43,} &{0,} &{0,} &{61,} &{67,} &{73,} &{79,} &{0,}  & {0,} &{97,} &{103,} & \ldots  \}, \\
{${\mathcal R}$  = } \{ &{1,} &{2,} &{3,} &{0,} &{5,} &{6,} &{7,} &{0,} &{0,} &{10,} &{11,} &{12,} &{13,} &{0,}  & {0,} &{16,} &{17,} & \ldots  \}.
\end{tabular}     
\end{center}
бЯЕ МЕМСКЕБШЕ ВКЕМШ ОНЯКЕДНБЮРЕКЭМНЯРЕИ ${\mathcal A}$ Х ${\mathcal B}$ ЪБКЪЧРЯЪ ОПНЯРШЛХ ВХЯКЮЛХ.
бЯЕ МЕМСКЕБШЕ ВКЕМШ ОНЯКЕДНБЮРЕКЭМНЯРЕИ ${\mathcal L}$ Х ${\mathcal R}$ ЪБКЪЧРЯЪ ХМДЕЙЯЮЛХ МЕМСКЕБШУ ВКЕМНБ ОНЯКЕДНБЮРЕКЭМНЯРЕИ ${\mathcal A}$ Х ${\mathcal B}$. 
нВЕБХДМН, ОНЯКЕДНБЮРЕКЭМНЯРХ $\mathcal L$ Х $\mathcal R$ МЮЯКЕДСЧР ЯРПСЙРСПС ОНЯКЕДНБЮРЕКЭМНЯРЕИ $\mathcal A$ Х $\mathcal B$ Б ЯЛШЯКЕ ПЮЯОПЕДЕКЕМХЪ МСКЕИ.  
хМДЕЙЯШ МСКЕБШУ ВКЕМНБ ОНЯКЕДНБЮРЕКЭМНЯРЕИ ${\mathcal A}$ Х ${\mathcal L}$ НОПЕДЕКЪЧРЯЪ ОПЮБШЛХ ВЮЯРЪЛХ БШПЮФЕМХИ (\ref{a_a_composite}) Х 
(\ref{a_b_composite}); 
ХМДЕЙЯШ МСКЕБШУ ВКЕМНБ ОНЯКЕДНБЮРЕКЭМНЯРЕИ ${\mathcal B}$ Х ${\mathcal R}$ НОПЕДЕКЪЧРЯЪ ОПЮБШЛХ ВЮЯРЪЛХ БШПЮФЕМХИ (\ref{b_a_composite}) Х 
(\ref{b_b_composite}). 

дЮКЕЕ НОПЕДЕКХЛ ОНЯКЕДНБЮРЕКЭМНЯРХ 
${\mathcal A}^{ m'} = \{  a_i^{ m'} \colon a_i^{ m'} =  a_{i + m'} \}$,
${\mathcal B}^{ m'} = \{  b_i^ { m'} \colon b_i^{ m'} =  b_{i + m'} \}$, 
ЪБКЪЧЫХЕЯЪ НЯРЮРЙЮЛХ ОНЯКЕДНБЮРЕКЭМНЯРЕИ ${\mathcal A}$ Х ${\mathcal B}$ ОНЯКЕ $m'$-ЦН ЩКЕЛЕМРЮ, МЮОПХЛЕП,
\tabcolsep=0.15em
 \begin{center}
\begin{tabular} {rrrrrrrrrrrrrrrrrrrrrr}
{${\mathcal A}^{\, 5}$ = } \{&{0,}&{41,}&{47,}&{53,}&{59,} &{0,} &{71,} &{0,} &{83,} &{89,} &{0,} &{101,}&{107,}&{113,} &{0,} &{0,}&{131,} &  \ldots \}, \\
{${\mathcal B}^{\, 6}$ = } \{&{43,} &{0,}&{0,}&{61,}&{67,} &{73,} &{79,} &{0,} &{0,} &{97,} &{103,} &{109,} &{0,} &{0,} &{127,} &{0,} &{139,} & \ldots \}.
\end{tabular}
\end{center} 
рЮЙХЛ ФЕ НАПЮГНЛ, ДКЪ ОНЯКЕДНБЮРЕКЭМНЯРЕИ ${\mathcal L}^{ m'}$ Х ${\mathcal R}^{ m'}$ АСДЕЛ ХЛЕРЭ
\tabcolsep=0.15em
 \begin{center}
\begin{tabular} {rrrrrrrrrrrrrrrrrrrrrr}
{${\mathcal L}^{\, 5}$ = } \{&{0,} &{7,} &{8,} &{9,}&{10,} &{0,} &{12,} &{0,} &{14,} &{15,} &{0,}  &{17,} &{18,} &{19,} &{0,} &{0,} &{22,} & \ldots \}, \\
{${\mathcal R}^{\, 6}$ = } \{&{7,} &{0,} &{0,} &{10,} &{11,} &{12,} &{13,} &{0,} & {0,} &{16,} &{17,} &{18,} &{0,} &{0,} &{21,} &{0,} &{23,} & \ldots \}.  
\end{tabular}
\end{center} 
б НАЫЕЛ ЯКСВЮЕ, БШПЮФЕМХЕ $S\{ g \}$ НАНГМЮВЮЕР ЙНМЕВМСЧ ХКХ АЕЯЙНМЕВМСЧ ОНЯКЕДНБЮРЕКЭМНЯРЭ, ВКЕМШ ЙНРНПНИ МЕЙНРНПШЛ НАПЮГНЛ НОПЕДЕКЕМШ.

\subsection{нРПЕГЙХ ОНЯКЕДНБЮРЕКЭМНЯРЕИ}\label{seg_of_seq}

оСЯРЭ $S_{m} = \{ s_i \}_{i=1}^{i=m}$ НАНГМЮВЮЕР МЮВЮКЭМШИ НРПЕГНЙ ДКХМШ $m$ ОНЯКЕДНБЮРЕКЭМНЯРХ $S$, МЮОПХЛЕП, 
\tabcolsep=0.3em
 \begin{center}
\begin{tabular} {rrrrrrrrrrrrrrrr}
{${\mathcal A}_{14}$ } &= ({5,} & {11,} & {17,} & {23,} & {29,} & {0,} & {41,} & {47,} & {53,} & {59,} & { 0,} & {71,} & { 0,} & {83}),   \\
{${\mathcal B}_{14}$ } &= ({7,} & {13,} & {19,} & { 0,} & {31,} & {37,} & {43,} & {0,} & {0,} & {61,} & {67,} & {73,} & {79,} & { 0}).   
\end{tabular}
\end{center} 
рЮЙХЕ НРПЕГЙХ ЛШ МЮГШБЮЕЛ \textit{ОПЪЛШЛХ} НРПЕГЙЮЛХ, Ю НРПЕГЙХ 
\tabcolsep=0.295em
 \begin{center}
\begin{tabular}{rrrrrrrrrrrrrrrr}
{${\mathcal A'}_{14}$  } &= ({83,} & { 0,} & {71,} & { 0,} & {59,} & {53,} & {47,} &{41,} &{0,} &{29,} &{23,} &{17,} &{11,} & {5}), \\
{${\mathcal B'}_{14}$  } &= ({ 0,} & {79,} & {73,} & {67,} & {61,} & {0,} & {0,} &{43,} &{37,} &{31,} &{0,} &{19,} &{13,} & {7}),  \\ 
\end{tabular}
\end{center} 
- \textit{НАПЮРМШЛХ} НРПЕГЙЮЛХ.
рЮЙХЛ ФЕ НАПЮГНЛ ЛШ НОПЕДЕКЪЕЛ НРПЕГЙХ ОНЯКЕДНБЮРЕКЭМНЯРЕИ ${\mathcal L}_{m}$, ${\mathcal R}_{m}$, $\mathcal L'_{m} $ Х $\mathcal R'_{m}$:
\tabcolsep=0.3em
 \begin{center}
\begin{tabular}{lrrrrrrrrrrrrrrrr}
{${\mathcal L}_{14} $} &= ({ 1,} & {2,} & {3,} & {4,} & {5,} & {0,} & {7,} & {8,} & {9,} & {10,} & { 0,} & {12,} & { 0,} & {14}),   \\
{${\mathcal R}_{14} $} &= ({ 1,} & {2,} & {3,} & {0,} & {5,} & {6,} & {7,} & {0,} & {0,} & {10,} & {11,} & {12,} & {13,} & { 0}), \\
{${\mathcal L'}_{14}$} &= ({14,} & { 0,} & {12,} & { 0,} & {10,} & {9,} & {8,} &{7,} &{0,} &{5,} &{4,} &{3,} &{2,} & {1}), \\
{${\mathcal R'}_{14}$} &= ({ 0,} & {13,} & {12,} & {11,} & {10,} & {0,} & {0,} &{7,} &{6,} &{5,} &{0,} &{3,} &{2,} & {1}). 
\end{tabular}
\end{center}

\subsection{тСМЙЖХЪ ЯВЕРЮ МЕМСКЕБШУ ЩКЕЛЕМРНБ}

оСЯРЭ $\pi (a,n)$ - ЙНКХВЕЯРБН ОПНЯРШУ БХДЮ $6i-1$ ЛЕМЭЬХУ $n$ Х ОСЯРЭ $\pi (b,n)$ - ЙНКХВЕЯРБН ОПНЯРШУ БХДЮ $6i+1$ ЛЕМЭЬХУ ХКХ ПЮБМШУ $n+1$. 
оНКЮЦЮЪ $n=6m$, АСДЕЛ ХЛЕРЭ
$$\begin{array}{*{20}l}
   {\pi }\left(a, n \right) = {{\pi }\left(a, 6m \right) = \# {\mathcal A}_{m} = \# {\mathcal L}_{m} = \sum\limits_{l_i  \leqslant m, \, l_i  \ne 0} 1 ,} 	\\
	 {\pi }\left(b, n \right) = {{\pi }\left(b, 6m \right) = \# {\mathcal B}_{m} = \# {\mathcal R}_{m} = \sum\limits_{r_i  \leqslant m, \, r_i  \ne 0} 1 .}  
 \end{array} $$

\subsection{вЕРМШЕ ВХЯКЮ} \label{Even numbers}

оСЯРЭ ${\mathcal G} = \{\mathrm g \}$ - ЛМНФЕЯРБН БЯЕУ ОНКНФХРЕКЭМШУ ВЕРМШУ ВХЯЕК.  
пЮГНАЗЕЛ ЛМНФЕЯРБН ${\mathcal G} \setminus \{2\}$ МЮ РПХ МЕОЕПЕЯЕЙЮЧЫХЕЯЪ ЛМНФЕЯРБЮ (ЙКЮЯЯШ БШВЕРНБ ОН ЛНДСКЧ $6$) Х АСДЕЛ ОНКЮЦЮРЭ, ВРН   
${\mathcal G^1} = \{ \mathrm g_m^1 : \mathrm g_m^1 = 6m - 2 \}$, 
${\mathcal G^2}=  \{ \mathrm g_m^2 : \mathrm g_m^2 = 6m \}$, 
${\mathcal G^3}=  \{ \mathrm g_m^3 : \mathrm g_m^3 = 6m + 2\}$. 
йЮФДНЕ ЖЕКНЕ ВХЯКН $m \geqslant 1$ НОПЕДЕКЪЕР РПХ ОНЯКЕДНБЮРЕКЭМШУ ВЕРМШУ ВХЯКЮ $\mathrm g_m^1$, $\mathrm g_m^2$, $\mathrm g_m^3$, ОН НДМНЛС ХГ ЙЮФДНЦН ЙКЮЯЯЮ.

\subsection{яСЛЛХПНБЮМХЕ НРПЕГЙНБ}\label{sum_of_segments}

яСЛЛС (ПЮГМНЯРЭ) ДБСУ ОНЯКЕДНБЮРЕКЭМНЯРЕИ $S' = \{s'_i\}$ Х $S'' = \{s''_i\}$ НОПЕДЕКХЛ ЙЮЙ ОНЯКЕДНБЮРЕКЭМНЯРЭ
$S = \{s_i: s_i = s'_i \pm s''_i  \text{ ЕЯКХ } s'_i s''_i \ne 0; s_i =0  \text{ ЕЯКХ } s'_i s''_i = 0 \}$. 
рЮЙХЛ ФЕ НАПЮГНЛ НОПЕДЕКХЛ ЯСЛЛС (ПЮГМНЯРЭ) ДБСУ НРПЕГЙНБ$S'_m = \{s'_i\}_{i=1}^{i=m}$ Х $S''_m = \{s''_i\}_{i=1}^{i=m}$ 
ЙЮЙ НРПЕГНЙ $S_m = \{s_i: s_i = s'_i \pm s''_i  \text{  ЕЯКХ  } s'_i s''_i \ne 0; s_i = 0  \text{  ЕЯКХ  } s'_i s''_i = 0 \}$.  

лШ ЯВХРЮЕЛ, ВРН ОНЯКЕДНБЮРЕКЭМНЯРЭ $S = S' + S''$ ОПНЯЕЪМЮ ОН $p$, ЕЯКХ НАЕ ОНЯКЕДНБЮРЕКЭМНЯРХ $S'$ Х $S''$ ОПНЯЕЪМШ ОН $p$,  
РН ЕЯРЭ $S \setminus \lambda p = S' \setminus \lambda p + S'' \setminus \lambda p$.

\subsection{аХМЮПМШЕ ЮДДХРХБМШЕ ГЮДЮВХ}\label{binary_problem}

\subsubsection{оЮПШ ОПНЯРШУ ВХЯЕК Я ТХЙЯХПНБЮММНИ ПЮГМНЯРЭЧ} \label{primes_with_fixed_gap}

оСЯРЭ ${\pi}_{{\mathrm g}}(n)$ -- ВХЯКН ОПНЯРШУ ВХЯЕК $p$, МЕ ОПЕБНЯУНДЪЫХУ $n$ Х РЮЙХУ, ВРН  $p' = p + \mathrm g$ РЮЙФЕ ОПНЯРНЕ.  
йЮФДШИ ХГ ЙКЮЯЯНБ ВЕРМШУ ВХЯЕК ЛНФЕР АШРЭ ОПЕДЯРЮБКЕМ Б БХДЕ ПЮГМНЯРХ ДБСУ МЕВЕРМШУ ХГ $A$ Х/ХКХ $B$ ЯРПНЦН НОПЕДЕКЕММШЛ НАПЮГНЛ, Ю ХЛЕММН  
$\mathrm g_{m'}^1=a_{i+m'}-b_i$, $\mathrm g_{m'}^2=(a_{i+m'}-a_i)=(b_{i+m'}-b_i)$, Х $\mathrm g_{m'}^3=b_{i+m'}-a_i$. 
щРХ РНФДЕЯРБЮ ОНГБНКЪЧР ОНЯРПНХРЭ ЙНМЯРПСЙЖХХ, ДЮЧЫХЕ ПЕЬЕМХЕ ГЮДЮВХ МЮУНФДЕМХЪ БЯЕУ РЮЙХУ ОЮП ДКЪ КЧАШУ  $\mathrm g > 2$ Х $n$:
\begin{align}
{\pi}_{{\mathrm g}^1_{m'}} (n+1) &= \# \left( {\mathcal A}^{m'}_{m}-{\mathcal B}_{m} \right) = \# \left( {\mathcal L}^{m'}_{m}-{\mathcal R}_{m} \right), \label{a-b} \\   
{\pi}_{{\mathrm g}^2_{m'}} (n+1) &= \# \left( {\mathcal A}^{m'}_{m}-{\mathcal A}_{m} \right) + \# \left( {\mathcal B}^{m'}_{m}-{\mathcal B}_{m} \right) \label{a-a}  \\  
	&= \# \left( {\mathcal L}^{m'}_{m} -  {\mathcal L}_{m} \right)  +  \# \left( {\mathcal R}^{m'}_{m} - {\mathcal R}_{m} \right), \nonumber  \\ 
  {\pi}_{{\mathrm g}^3_{m'}} (n) &= \# \left( {\mathcal B}^{m'}_{m} -{\mathcal A}_{m} \right)  = \# \left( {\mathcal R}^{m'}_{m}- {\mathcal L}_{m} \right). \label{b-a}
  \end{align}
гЮЛЕРХЛ, ВРН ${\pi}_{{\mathrm g}^1_{m'}} (n+1) - {\pi}_{{\mathrm g}^1_{m'}} (n) \leqslant 1$ Х ${\pi}_{{\mathrm g}^2_{m'}} (n+1) - {\pi}_{{\mathrm g}^2_{m'}} (n) \leqslant 1$.

\begin{remark} \label{rem_on_diff}
еЯКХ $3 \mid \mathrm g$, РН $\mathrm g$ ХЛЕЕР ОПЕДЯРЮБКЕМХЪ ДБСУ БХДНБ, ЙЮЙ ОНЙЮГЮМН Б (\ref{a-a}).
\end{remark}

мЮИДЕЛ, МЮОПХЛЕП, ВХЯКН ОЮП $p, p' = p + 28$, ЦДЕ $p \leqslant 126 + 1 $.  
лШ ХЛЕЕЛ $28 = \mathrm g_5^1 \in \mathcal G^1$, $m' = 5$ Х $m = 126/6 = 21$. оПХЛЕМЪЪ ЙНМЯРПСЙЖХЧ (\ref{a-b}), ОНКСВХЛ  
\tabcolsep=0.085em
\begin{center}
\begin{tabular}{rrrrrrrrrrrrrrrrrrrrrrrrrrr}
 ${\mathcal A}^5_{21} =$ &{0,}&{41,}&{47,}&{53,}&{59,}&{0,}&{71,}&{0,}&{83,}&{89,}&{0,}&{101,}&{107,}&{113,}&{0,}&{0,}&{131,}&{137,}&{0,}&{149,}&{0}. \\ 
   ${\mathcal B}_{21} =$ & {7,} &{13,}&{19,}&{0,}&{31,}&{37,}&{43,}&{0,}&{0,}&{61,}&{67,}&{73,}&{79,}&{0,}&{0,}&{97,}&{103,}&{109,}&{0,}& {0,}&{127}. \\  
 \hline 
$S_{21}\{28^-\} \setminus \lambda \mathcal P =$ &{0,}&{28,}&{28,}&{0,}&{28,}&{0,}&{28,}&{0,}&{0,}&{28,}&{0,}&{28,}&{28,}&{0,}&{0,}&{0,}&{28,}&{28,}&{0,}
                                               &{0,}&{0}.                                                        
\end{tabular}
\end{center}
яХЛБНК $28^-$ НАНГМЮВЮЕР, ВРН ВХЯКН $28$ ГДЕЯЭ ОПЕДЯРЮБКЕМН ЙЮЙ ПЮГМНЯРЭ ДБСУ ОПНЯРШУ ВХЯЕК.
йЮФДШИ МЕМСКЕБНИ ЩКЕЛЕМР НРПЕГЙЮ $({ S_{21} \{28^-\} \setminus \lambda \mathcal P })$ НОПЕДЕКЪЕР НДМН ХГ ОПЕДЯРЮБКЕМХИ ВХЯКЮ $\mathrm g_{5}^1 = 28$ Б БХДЕ ПЮГМНЯРХ ДБСУ ОПНЯРШУ.
рЮЙХЛ НАПЮГНЛ, $\pi_{28} (127) = \# ({S_{21} \{28^-\} \setminus \lambda \mathcal P }) = 9$. 

оНДЯРЮБХБ $m' = 5$ Х $m = 126/6 = 21$ Б ЙНМЯРПСЙЖХХ (\ref{a-a}) Х (\ref{b-a}), ЛШ ОНКСВХЛ ВХЯКН ОЮП ОПНЯРШУ Я ПЮГМНЯРЭЧ $30$ Х $32$ ЯННРБЕРЯРБЕММН.

\subsubsection{оПНЯРШЕ ВХЯКЮ АКХГМЕЖШ} \label{twinprimes}

нЯНАШЛ Х БЮФМЕИЬХЛ ЯКСВЮЕЛ ОПНЯРШУ Я ТХЙЯХПНБЮММНИ ПЮГМНЯРЭЧ ЪБКЪЧРЯЪ ОПНЯРШЕ-АКХГМЕЖШ.
хГ РНФДЕЯРБЮ $ b_{i}-a_i = 2 $ ОНКСВЮЕЛ ЙНМЯРПСЙЖХЧ
$${\pi}_2 (n)  = \# \left( {\mathcal B}_{m}  -  {\mathcal A}_{m} \right) = \# \left( {\mathcal R}_{m}  -  {\mathcal L}_{m} \right)$$ 
ДКЪ ВХЯКЮ ОЮП ОПНЯРШУ-АКХГМЕЖНБ $\pi_2 (n)$, МЕ ОПЕБНЯУНДЪЫХУ $n = 6m$.
оСЯРЭ $B - A = T$.
рЮЙ ЙЮЙ ${\mathcal A} = A {\setminus \lambda \mathcal P}$, ${\mathcal B} = B {\setminus \lambda \mathcal P}$, РН 
\tabcolsep=0.092em
\begin{center}
\begin{tabular}{rrrrrrrrrrrrrrrrrrrrrrrrrrr}
${{\mathcal B}}$ = &{7,}&{13,}&{19,}&{0,} &{31,}&{37,}&{43,}&{0,} &{0,} &{61,}&{67,}&{73,}&{79,}&{0,} &{0,} &{97,}&{103,}&{109,}&{0,}  &{0,}&{127, \ldots} \\ 
${{\mathcal A}}$ = &{5,}&{11,}&{17,}&{23,}&{29,}&{0,} &{41,}&{47,}&{53,}&{59,}&{0,} &{71,}&{0,} &{83,}&{89,}&{0,} &{101,}&{107,}&{113,}&{0,}&{0, \ldots}   \\ 
\hline
${{T} \setminus \lambda \mathcal P}$ = &{2,}&{2,}&{2,}&{0,}&{2,}&{0,}&{2,}&{0,}&{0,}&{2,}&{0,}&{2,}&{0,}&{0,}&{0,}&{0,}&{2,}&{2,} &{0,}  &{0,}&{0, \ldots}   
\end{tabular}
\end{center}
б ОНЯКЕДНБЮРЕКЭМНЯРХ ${{T} \setminus \lambda \mathcal P}$ ЙЮФДШИ ВКЕМ ПЮБМШИ $2$ СЙЮГШБЮЕР МЮ ОЮПС ОПНЯРШУ-АКХГМЕЖНБ.  %
рЕОЕПЭ НОПЕДЕКХЛ БЮФМСЧ ОНЯКЕДНБЮРЕКЭМНЯРЭ
$ {\mathcal T} = \left\{  t_i \colon {t_i} = i \text{  if  } {l_i}  {r_i} \ne 0 ; {t_i} = 0  \text{  if  } {l_i}  {r_i} = 0 \right\}$: 
\tabcolsep=0.094em
\begin{center}
\begin{tabular}{rrrrrrrrrrrrrrrrrrrrrrrrrrr}
${{\mathcal R}}= $ &{1,}&{2,}&{3,}&{0,}&{5,}&{6,}&{7,}&{0,}&{0,}&{10,}&{11,}&{12,}&{13,}&{0,}&{0,}&{16,}&{17,}&{18,}&{0,}&{0,}&{21,}&{0,}&{23,}&{0,}&{25 \ldots} \\ 
${{\mathcal L}} =$ &{1,}&{2,}&{3,}&{4,}&{5,}&{0,}&{7,}&{8,}&{9,}&{10,}&{0,}&{12,}&{0,}&{14,}&{15,}&{0,}&{17,}&{18,}&{19,}&{0,}&{0,}&{22,}&{23,}&{0,}&{25 \ldots} \\
\hline
${{\mathcal T}}= $ &{1,}&{2,}&{3,}&{0,}&{5,}&{0,}&{7,}&{0,}&{0,}&{10,}&{0,}&{12,}&{0,}&{0,}&{0,}&{0,}&{17,}&{18,}&{0,}&{0,}&{0,}&{0,}&{23,}&{0,}&{25 \ldots}  \\
\end{tabular}
\end{center}
гДЕЯЭ, ЕЯКХ $t_i \ne 0$, РНЦДЮ $6t_i \mp 1$ НАЮ ОПНЯРШЕ, РН ЕЯРЭ НАПЮГСЧР ОЮПС ОПНЯРШУ-АКХГМЕЖНБ.
кЕЦЙН БХДЕРЭ, ВРН ОНЯКЕДНБЮРЕКЭМНЯРЭ $\mathcal T$ МЮЯКЕДСЕР ЯРПСЙРСПС $T \setminus \lambda \mathcal P$ Б ЯЛШЯКЕ ПЮЯОПЕДЕКЕМХЪ МСКЕИ.

\subsubsection{оПЕДЯРЮБКЕМХЕ ВЕРМШУ ВХЯЕК Б БХДЕ ЯСЛЛШ ДБСУ ОПНЯРШУ}\label{reprofevennumber}

оСЯРЭ $\pi^+ (\mathrm g)$ ВХЯКН ОПЕДЯРЮБКЕМХИ ВЕРМНЦН ВХЯКЮ $\mathrm g$ Б БХДЕ ЯСЛЛШ ДБСУ ОПНЯРШУ.
йЮФДШИ ХГ ЙКЮЯЯНБ ВЕРМШУ ВХЯЕК ЛНФЕР АШРЭ ОПЕДЯРЮБКЕМ Б БХДЕ ЯСЛЛШ ДБСУ МЕВЕРМШУ ХГ $A$ Х/ХКХ $B$ ЯРПНЦН НОПЕДЕКЕММШЛ НАПЮГНЛ, Ю ХЛЕММН: 
$\mathrm g_{m+1}^1=a_i+a_{m-i+1}$, $\mathrm g_{m+1}^2=a_i+b_{m-i+1}$ Х $\mathrm g_{m+1}^3=b_j+b_{m-j+1}$. 
щРХ РНФДЕЯРБЮ ОНГБНКЪЧР ОНЯРПНХРЭ ЙНМЯРПСЙЖХХ, ДЮЧЫХЕ ПЕЬЕМХЕ ГЮДЮВХ МЮУНФДЕМХЪ БЯЕУ РЮЙХУ ОЮП ДКЪ БЯЕУ ВЕРМШУ ВХЯЕК $\mathrm g \geqslant 10$: 
\begin{align}
\pi^+ (\mathrm g^1_{m}) & \geqslant 0.5 \cdot \# \left( {\mathcal A}_{m-1} + {\mathcal A'}_{m-1} \right) = 0.5 \cdot \# \left( {\mathcal L}_{m-1} + {\mathcal L'}_{m-1} \right), \label{a+a} \\ 
\pi^+ (\mathrm g^2_{m}) & = \# \left( {\mathcal A}_{m-1}+{\mathcal B'}_{m-1} \right) = \# \left({\mathcal L}_{m-1} +{\mathcal R'}_{m-1} \right), \label{a+b} \\ 
    \pi^+ (\mathrm g^3_{m}) & \geqslant 0.5 \cdot \# \left( {\mathcal B}_{m-1} + {\mathcal B'}_{m-1} \right) = 0.5 \cdot \# \left( {\mathcal R}_{m-1} + {\mathcal R'}_{m-1} \right).  \label{b+b}  
\end{align}
\begin{remark} \label{rem_on_sum}
йНЩТТХЖХЕМР $0.5$ Б (\ref{a+a}) Х (\ref{b+b}) НАСЯКНБКЕМ РЕЛ, ВРН ЯСЛЛЮ \textit{ОПЪЛНЦН} Х \textit{НАПЮРМНЦН} НРПЕГЙНБ НДМНЦН Х РНЦН ФЕ ЙКЮЯЯЮ ДЮЕР НРПЕГНЙ, ЩКЕЛЕМРШ ЙНРНПНЦН, ПЮЯОНКНФЕММШЕ ЯХЛЛЕРПХВМН НРМНЯХРЕКЭМН ЖЕМРПЮ, НРКХВЮЧРЯЪ КХЬЭ ОНПЪДЙНЛ ЯКЮЦЮЕЛШУ. 
\end{remark}

мЮИДЕЛ, МЮОПХЛЕП, ВХЯКН ОПЕДЯРЮБКЕМХИ ВЕРМНЦН ВХЯКЮ $94$ Б БХДЕ ЯСЛЛШ ДБСУ ОПНЯРШУ.
гДЕЯЭ ЛШ ХЛЕЕЛ $94 = \mathrm g_{16}^1 \in \mathcal G^1$, $m = 16$. оПХЛЕМЪЪ ЙНМЯРПСЙЖХЧ (\ref{a+a}), ОНКСВХЛ
\tabcolsep=0.32em
\begin{center}
\begin{tabular}{rrrrrrrrrrrrrrrrr}
  ${\mathcal A}_{15}  $  =  & {5,}  & {11,} & {17,} & {23,} & {29,} &{0,}  &{41,} &{47,} &{53,} &{59,} &{0,}  &{71,} &{0,}  &{83,}  &{89.}   \\
  ${\mathcal A'}_{15} $  =  & {89,} &{83,}  & {0,}  & {71,} & {0,}  &{59,} &{53,} &{47,} &{41,} &{0,}  &{29,} &{23,} &{17,} &{11,}  &{5.}   \\
\hline
  $S_{15}\{94^+\} \setminus \lambda \mathcal P$  = &{94,} &{94,} &{0,} &{94,}  &{0,}  &{0,}  &{94,} &{94,} &{94,}  &{0,} &{0,}  &{94,} &{0,}  &{94,} &{94.}					
\end{tabular}
\end{center}
оНКСВЕММНЕ ГМЮВЕМХЕ $0.5 \cdot \# (S_{15}\{94^+\} \setminus \lambda \mathcal P)  = 4.5$ Б РН БПЕЛЪ ЙЮЙ ХЯРХММНЕ ГМЮВЕМХЕ $\pi^+ (94) = 5 $. 
щРНР ФЕ ПЕГСКЭРЮР ЛШ ОНКСВХЛ ОПХЛЕМЪЪ ОНЯКЕДМЕЕ БШПЮФЕМХЕ ХГ (\ref{a+a}): 
\tabcolsep=0.38em
\begin{center}
\begin{tabular}{rrrrrrrrrrrrrrrrr}
     ${\mathcal L}_{15}  $ =  &{1,}  &{2,}  &{3,} &{4,}  &{5,} &{0,}  &{7,} &{8,}  &{9,}  &{10,} &{0,} &{12,} &{0,} &{14,} &{15.} \\
     ${\mathcal L'}_{15} $ =  &{15,} &{14,} &{0,} &{12,} &{0,} &{10,} &{9,} &{8,}  &{7,}  &{0,}  &{5,} &{4,}  &{3,} &{2,}  &{1.}   \\
\hline
  $S_{15}\{16^+\} \setminus \lambda \mathcal P$  =&{16,} &{16,} &{0,} &{16,} &{0,} &{0,} &{16,} &{16,} &{16,} &{0,}  &{0,} &{16,} &{0,} &{16,} &{16.}
\end{tabular}
\end{center}

йНМЯРПСЙЖХХ (\ref{a+b}) Х (\ref{b+b}) ОПХ $m = 16$ ДЮДСР ЯННРБЕРЯРБЕММН ВХЯКН ОПЕДЯРЮБКЕМХИ ВЕРМШУ ВХЯЕК $96$ and $98$ Б БХДЕ ЯСЛЛШ ДБСУ ОПНЯРШУ.  
\begin{remark}\label{rem_on_density}
хГ ЯСЫЕЯРБНБЮМХЪ ОПЕДЯРЮБКЕМХЪ Б БХДЕ ЯСЛЛШ ДБСУ ОПНЯРШУ ДКЪ ЙЮФДНЦН ХГ РПЕУ ОНЯКЕДНБЮРЕКЭМШУ ВЕРМШУ ВХЯЕК $\mathrm g_{m}^1$, $\mathrm g_{m}^2$ Х $\mathrm g_{m}^3$ ЯКЕДСЕР, ВРН ХЛЕЕРЯЪ ОН ЙПЮИМЕИ ЛЕПЕ НДМН ДКЪ ЙЮФДНЦН ХГ РПЕУ ОПЕДЯРЮБКЕМХИ ВХЯКЮ $m$ Б БХДЕ $m=l+l$, $m=l+r$ Х $m=r+r$.
\end{remark}

\subsection{дБНИМНЕ ПЕЬЕРН}\label{double_sieving} 

б ОЮПЮЦПЮТЮУ \ref{primes_and_composite}, \ref{sieve_method}, \ref{sequences L and R} ЛШ ОНЙЮГЮКХ, ВРН ЙЮФДЮЪ ХГ ОНЯКЕДНБЮРЕКЭМНЯРЕИ $A$ Х $B$ УНПНЬН ЯРПСЙРСПХПНБЮМЮ ОН БЯЕЛ ОПНЯРШЛ $p \in \mathcal{P}$.
рЮЙХЛ НАПЮГНЛ, ДКЪ КЧАНЦН $p \in \mathcal{P}$ Б ЙЮФДНИ ХГ ОНЯКЕДНБЮРЕКЭМНЯРЕИ ${\mathcal A}$, ${\mathcal B}$, ${\mathcal L}$ Х ${\mathcal R}$ МЮИДЕРЯЪ НДМЮ Х РНКЭЙН НДМЮ АЕЯЙНМЕВМЮЪ ОНДОНЯКЕДНБЮРЕКЭМНЯРЭ МСКЕБШУ ВКЕМНБ, ХМДЕЙЯШ ЙНРНПШУ НАПЮГСЧР ЮПХТЛЕРХВЕЯЙСЧ ОПНЦПЕЯЯХЧ Я ПЮГМНЯРЭЧ $p$.  
дКЪ ЙЮФДНЦН $p \in \mathcal{P}$ ЛШ ХЛЕЕЛ
\begin{align*}
\# (A_{m} {\setminus \lambda p}) = &\left| \left\{ a_i \colon i \leqslant m, \left( p, a_i \right) = 1  \right\}  \right| \sim m \left( 1 - 1/p \right), \\
\# (B_{m} {\setminus \lambda p}) = &\left| \left\{ b_i \colon i \leqslant m, \left( p, b_i \right) = 1  \right\}  \right| \sim m \left( 1 - 1/p \right).
\end{align*}
\begin{definition} 
оНЯКЕДНБЮРЕКЭМНЯРЭ $S$ ДБЮФДШ ОПНЯЕЪМЮ ОН ОПНЯРНЛС ВХЯКС $p$, ЕЯКХ НМЮ ЯНДЕПФХР ДБЕ ПЮГКХВМШЕ (МЕЯНБОЮДЮЧЫХЕ) ОНДОНЯКЕДНБЮРЕКЭМНЯРХ МСКЕБШУ ВКЕМНБ, ХМДЕЙЯШ ЙНРНПШУ НАПЮГСЧР ЮПХТЛЕРХВЕЯЙХЕ ОПНЦПЕЯЯХХ Я ПЮГМНЯРЭЧ $p$.
\end{definition} 
\begin{definition} 
мЮ ОНЯКЕДНБЮРЕКЭМНЯРХ $S$ ПЕЮКХГНБЮМН ДБНИМНЕ ПЕЬЕРН, ЕЯКХ ОНЯКЕДНБЮРЕКЭМНЯРЭ $S$ ДБЮФДШ ОПНЯЕЪМЮ ОН БЯЕЛ $p\in \mathcal P$; 
\end{definition} 
\begin{definition} 
мЮ ОНЯКЕДНБЮРЕКЭМНЯРХ $S$ ПЕЮКХГНБЮМН ДБНИМНЕ ПЕЬЕРН, ЕЯКХ ОНЯКЕДНБЮРЕКЭМНЯРЭ $S$ ДБЮФДШ ОПНЯЕЪМЮ ОН БЯЕЛ $p \in \mathcal P$, ЙПНЛЕ ЙНМЕВМНЦН ХУ ВХЯКЮ, ОН ЙНРНПШЛ ОНЯКЕДНБЮРЕКЭМНЯРЭ $S$ ОПНЯЕЪМЮ НДМНЙПЮРМН (НАЫХИ ЯКСВЮИ).  
\end{definition} 

\begin{theorem} 
оНЯКЕДНБЮРЕКЭМНЯРХ ${{T} \setminus \lambda \mathcal P}$ Х $\mathcal T$ ЕЯРЭ ПЕЮКХГЮЖХХ ДБНИМНЦН ПЕЬЕРЮ.
\end{theorem} 

\begin{proof}
оНЯКЕДНБЮРЕКЭМНЯРХ ${{T} \setminus \lambda \mathcal P}$ Х $\mathcal T$ ХЛЕЧР НДХМЮЙНБСЧ ЯРПСЙРСПС Б ЯЛШЯКЕ ПЮЯОПЕДЕКЕМХЪ МСКЕБШУ ВКЕМНБ.
пЮЯЯЛНРПХЛ ЯРПСЙРСПС ОНЯКЕДНБЮРЕКЭМНЯРХ ${{T} \setminus \lambda \mathcal P}$:
\tabcolsep=0.092em
\begin{center}
\begin{tabular}{rrrrrrrrrrrrrrrrrrrrrrrrrrr}
${{\mathcal B}}$ = &{7,}&{13,}&{19,}&{0,} &{31,}&{37,}&{43,}&{0,} &{0,} &{61,}&{67,}&{73,}&{79,}&{0,} &{0,} &{97,}&{103,}&{109,}&{0,}  &{0,}&{127, \ldots} \\ 
${{\mathcal A}}$ = &{5,}&{11,}&{17,}&{23,}&{29,}&{0,} &{41,}&{47,}&{53,}&{59,}&{0,} &{71,}&{0,} &{83,}&{89,}&{0,} &{101,}&{107,}&{113,}&{0,}&{0, \ldots}   \\ 
\hline
${{T} \setminus \lambda \mathcal P}$ = &{2,}&{2,}&{2,}&{0,}&{2,}&{0,}&{2,}&{0,}&{0,}&{2,}&{0,}&{2,}&{0,}&{0,}&{0,}&{0,}&{2,}&{2,}&{0,}&{0,}&{0, \ldots}   
\end{tabular}
\end{center}
вКЕМШ НАЕХУ ОНЯКЕДНБЮРЕКЭМНЯРЕИ ${\mathcal A}$ Х ${\mathcal B}$, ОПННАПЮГЮЛХ ЙНРНПШУ ЪБКЪЧРЯЪ ЯНЯРЮБМШЕ ВХЯКЮ Х ЙНРНПШЕ НАНГМЮВЕМШ ЙЮЙ $0$, НРНАПЮФЮЧРЯЪ МЮ ${{T} \setminus \lambda \mathcal P}$ РЮЙФЕ ЙЮЙ $0$.
рЮЙХЛ НАПЮГНЛ, ХМДЕЙЯШ МСКЕБШУ ВКЕМНБ ОНЯКЕДНБЮРЕКЭМНЯРХ ${{T} \setminus \lambda \mathcal P}$ АСДСР НОПЕДЕКЪРЭЯЪ БЯЕЛХ БШПЮФЕМХЪЛХ (\ref{a_a_composite}), (\ref{a_b_composite}), (\ref{b_a_composite}) Х (\ref{b_b_composite}).
нАЗЕДХМЪЪ (\ref{a_a_composite}) Я (\ref{b_a_composite}) Х (\ref{a_b_composite}) Я (\ref{b_b_composite}), ЛШ ОНКСВХЛ ВЕРШПЕ ЯЕЛЕИЯРБЮ ЮПХТЛЕРХВЕЯЙХУ ОПНЦПЕЯЯХИ, $\lambda \in  \textbf{Z}^+$,
\begin{align} 
	 i =   \lambda a_j \pm j , \label{a_a_composite_2} \\ 
	 i =   \lambda b_k \pm k . \label{b_b_composite_2}  
	\end{align}
хГ БШПЮФЕМХЪ (\ref{a_a_composite_2}) ЯКЕДСЕР, ВРН Б ОНЯКЕДНБЮРЕКЭМНЯРХ ${{T} \setminus \lambda \mathcal P}$ ДКЪ ЙЮФДНЦН $p \in A$ ХЛЕЕРЯЪ ДБЕ ПЮГКХВМШЕ ОНДОНЯКЕДНБЮРЕКЭМНЯРХ МСКЕБШУ ВКЕМНБ, ХМДЕЙЯШ ЙНРНПШУ НАПЮГСЧР ЮПХТЛЕРХВЕЯЙХЕ ОПНЦПЕЯЯХХ Я ПЮГМНЯРЭЧ $p$.
хГ БШПЮФЕМХЪ (\ref{b_b_composite_2}) ЯКЕДСЕР, ВРН Б ОНЯКЕДНБЮРЕКЭМНЯРХ ${{T} \setminus \lambda \mathcal P}$ ДКЪ ЙЮФДНЦН $p \in B$ ХЛЕЕРЯЪ ДБЕ ПЮГКХВМШЕ ОНДОНЯКЕДНБЮРЕКЭМНЯРХ МСКЕБШУ ВКЕМНБ, ХМДЕЙЯШ ЙНРНПШУ НАПЮГСЧР ЮПХТЛЕРХВЕЯЙХЕ ОПНЦПЕЯЯХХ Я ПЮГМНЯРЭЧ $p$.
нЙНМВЮРЕКЭМН, СРБЕПФДЕМХЕ РЕНПЕЛШ ЯКЕДСЕР ХГ РНЦН, ВРН $A \cup B \supset  \mathcal P$.
\end{proof}
рЮЙХЛ НАПЮГНЛ, ЯСЛЛЮ ОНЯКЕДНБЮРЕКЭМНЯРЕИ $\mathcal A$ Х $\mathcal B$, Ю РЮЙФЕ, ЯСЛЛЮ ОНЯКЕДНБЮРЕКЭМНЯРЕИ  ${\mathcal L}$ Х ${\mathcal R}$ ЪБКЪЧРЯЪ ПЕЮКХГЮЖХЪЛХ ДБНИМНЦН ПЕЬЕРЮ.
дКЪ КЧАНЦН $p \in \mathcal P$ АСДЕЛ ХЛЕРЭ
\begin{align*}
\# ({T_m \setminus \lambda p}) \sim m \left( 1 - 2/p \right).	 
\end{align*}

\subsubsection{нАЫХИ ЯКСВЮИ ДБНИМНЦН ПЕЬЕРЮ}\label{general_sieving}
бЯЕ АХМЮПМШЕ ЙНМЯРПСЙЖХХ, ПЮЯЯЛНРПЕММШЕ Б \ref{primes_with_fixed_gap}, \ref{twinprimes} Х \ref{reprofevennumber} ХЛЕЧР НАЫХИ БХД $S'_{m} \pm S''_{m}={S}_{m} = \{ \mathrm g \}_{i=1}^{i=m}$ Х НЯМНБЮМШ МЮ РНФДЕЯРБЕ $s'_i \pm s''_i = s_i =\mathrm g$. 
нРЯЧДЮ, ЕЯКХ $q \mid  \mathrm g$, РН $q \mid s'_i$ РНЦДЮ Х РНКЭЙН РНЦДЮ, ЙНЦДЮ $q \mid s''_i$.
рЮЙХЛ НАПЮГНЛ, НАЮ ЩКЕЛЕМРЮ $s'_i$ Х $s''_i$ НРНАПЮФЮЧРЯЪ Б НДХМ Х РНР ФЕ ЩКЕЛЕМР $s_i$.
оСЯРЭ $p, q \in \mathcal{P}$ Х ОСЯРЭ $q \mid  \mathrm g$ Х $p \nmid  \mathrm g$; РНЦДЮ
\begin{align}
\# ({S_m \setminus \lambda q}) \sim  m \left( 1 - 1/q \right), \label{one_sieving} \\ 
\# ({S_m \setminus \lambda p}) \sim  m \left( 1 - 2/p \right). \label{two_sieving} 
	\end{align}
дЮДХЛ ДБЮ ОПХЛЕПЮ ДКЪ ХККЧЯРПЮЖХХ ЩРХУ СРБЕПФДЕМХИ. 
пЮЯЯЛНРПХЛ ЙНМЯРПСЙЖХЧ $\{ 28^{-} \} = A^5 - B $. пЮГМНЯРЭ ЛЕФДС ЯННРБЕРЯРБСЧЫХЛХ ВКЕМЮЛХ ОНЯКЕДНБЮРЕКЭМНЯРЕИ $A^5$ Х $B$ ПЮБМЮЪ $\mathrm g=28$ ХЛЕЕР ЕДХМЯРБЕММШИ ОПНЯРНИ ДЕКХРЕКЭ $p=7$.
оПНЯЕЪБ ЩРС ЙНМЯРПСЙЖХЧ ОН ОПНЯРНЛС ВХЯКС $p=7$, $7 \mid 28$, ЛШ ОНКСВХЛ

\tabcolsep=0.08em
\begin{center}
\begin{tabular}{rrrrrrrrrrrrrrrrrrrrrrrrrrr}
${A}^5 \setminus \lambda 7 =$ &{0,}&{41,}&{47,}&{53,}&{59,}&{65,} &{71,}&{0,}&{83,}&{89,}&{95,}&{101,}&{107,}&{113,}&{0,}&{125,}&{131,}&{137,}&{143, \ldots} \\ 
${B}\setminus \lambda 7 =$  &{7,}&{13,}&{19,}&{25,}&{31,}&{37,}&{43,}&{0,}&{55,}&{61,}&{67,}&{73,}&{79,}&{85,}&{0,}&{97,}&{103,}&{109,}&{115, \ldots}  \\ 
\hline 
${S\{ 28^- \}} \setminus \lambda 7 =$ &{0,}&{28,}&{28,}&{28,}&{28,}&{28,}&{28,}&{0,}&{28,}&{28,}&{28,}&{28,}&{28,}&{28,}&{0,}&{28,}&{28,}&{28,}&{28, \ldots}   
\end{tabular}
\end{center}
ЦДЕ Б ЯННРБЕРЯРБХЕ Я (\ref{one_sieving}) Б ОНЯКЕДНБЮРЕКЭМНЯРХ $S\{ 28^- \setminus \lambda 7\}$ БШПЕГЮМ НДХМ ХГ ЙЮФДШУ ЯЕЛХ ВКЕМНБ. 
оПНЯЕЪБ ЩРС ФЕ ЙНМЯРПСЙЖХЧ ОН ОПНЯРНЛС ВХЯКС $p=5$, $5 \nmid 28$, ОНКСВХЛ
\tabcolsep=0.078em
\begin{center}
\begin{tabular}{rrrrrrrrrrrrrrrrrrrr}
${A}^5 \setminus \lambda 5 =$ &{0,}&{41,}&{47,}&{53,}&{59,}&{0,}&{71,}&{77,}&{83,}&{89,}&{0,}&{101,}&{107,}&{113,}&{119,}&{0,}&{131,}&{137,}&{143, \ldots} \\ 
${B} \setminus \lambda 5 =$ & {7,}&{13,}&{19,}&{0,}&{31,}&{37,}&{43,}&{49,}&{0,}&{61,}&{67,}&{73,}&{79,}&{0,}&{91,}&{97,}&{103,}&{109,}&{0, \ldots} \\ 
\hline 
${S\{ 28^- \}} \setminus \lambda 5=$ &{0,}&{28,}&{28,}&{0,}&{28,}&{0,}&{28,}&{28,}&{0,}&{28,}&{0,}&{28,}&{28,}&{0,}&{28,}&{0,}&{28,}&{28,}&{0, \ldots}   
\end{tabular}
\end{center} 
ЦДЕ Б ЯННРБЕРЯРБХЕ Я (\ref{two_sieving}) Б ОНЯКЕДНБЮРЕКЭМНЯРХ $S\{ 28^-  \setminus \lambda 5 \}$ БШПЕГЮМШ ДБЮ ХГ ЙЮФДШУ ОЪРХ ВКЕМНБ. 
рЮЙНИ ФЕ ПЕГСКЭРЮР ЛШ ОНКСВХЛ ДКЪ КЧАНЦН $p \in  \mathcal P$, $p \nmid 28$.

оПЕДОНКНФХЛ, ВРН ОПХ $n \to \infty$ ЯСЫЕЯРБСЕР НЖЕМЙЮ БХДЮ
  \begin{align*} 
   {\pi}_{2} (n)  \sim  mC \prod \limits_{5 \leqslant p \leqslant n} \left( {1 - \frac{2}{p}} \right). 
 \end{align*}
еЯКХ МХ НДМН $p \in  \mathcal P$ МЕ ДЕКХР $\mathrm g$ Х $3 \nmid  \mathrm g$ (РН ЕЯРЭ ДКЪ КЧАНЦН $\mathrm g = 2^{\nu}$), АСДЕЛ ХЛЕРЭ ОПХ $n \to \infty$ 
\begin{align*} 
   {\pi}_{\mathrm g} (n)  \sim  mC \prod \limits_{5 \leqslant p \leqslant n} \left( {1 - \frac{2}{p}} \right) 
	                       \prod \limits_{ n \leqslant p \leqslant n + \mathrm g } \left( {1 - \frac{1}{p}} \right). 
 \end{align*}
еЯКХ $p_k \mid  \mathrm g$, $p_k \in  \mathcal P$ Х $3 \nmid  \mathrm g$ (РН ЕЯРЭ, ЕЯКХ $\mathrm g = p_k 2^{\nu}$, $p_k \ll n$), РН  
\begin{align*} 
   {\pi}_{\mathrm g} (n)  \sim  mC \; \frac{p_k-1}{p_k-2}  \prod \limits_{5 \leqslant p \leqslant n} \left( {1 - \frac{2}{p}} \right) 
	                       \prod \limits_{ n \leqslant p \leqslant n + \mathrm g } \left( {1 - \frac{1}{p}} \right). 
 \end{align*}
б НАЫЕЛ ЯКСВЮЕ, ДКЪ ОПНХГБНКЭМНЦН ВЕРМНЦН $\mathrm g$ АСДЕЛ ХЛЕРЭ ОПХ $n \to \infty$
\begin{align}  \label{first_HLC} 
   {\pi}_{\mathrm g} (n)  \sim  m C  
	\kappa \prod_{\substack { p \mid \mathrm g}, p \in  \mathcal P}   \left(\frac{p - 1} {p-2} \right) 
	\prod \limits_{5 \leqslant p \leqslant n} \left( {1 - \frac{2}{p}} \right) 
	  \prod \limits_{ n \leqslant p \leqslant n + \mathrm g } \left( {1 - \frac{1}{p}} \right). 
 \end{align}
гДЕЯЭ, Б ЯННРБЕРЯРБХЕ Я гЮЛЕВЮМХЕЛ \ref{rem_on_diff} Й ОЮПЮЦПЮТС \ref{primes_with_fixed_gap}, ЯКЕДСЕР ОНКНФХРЭ $\kappa = 1$ ЕЯКХ $3 \nmid {\mathrm g}$ Х $\kappa = 2$ ЕЯКХ $3 \mid{\mathrm g}$. оНЯКЕДМХИ ЛМНФХРЕКЭ Б (\ref{first_HLC}) ОПХ КЧАНЛ ЙНМЕВМНЛ $\mathrm g$ ЯРПЕЛХРЯЪ Й $1$. рЮЙХЛ НАПЮГНЛ, ОПХ $n \to \infty$  ЛНФМН ЯВХРЮРЭ ОНЯРНЪММШЛ НРМНЬЕМХЕ
\begin{align*}  
    \eta_2 (\mathrm g) 
		= 	\frac {{\pi}_{\mathrm g} (n)} {{\pi}_2 (n)} 
		= \kappa \prod_{\substack { p \mid \mathrm g}, p \in  \mathcal P}   \left(\frac{p - 1} {p} : \frac{ p-2} {p}\right) 
		=  \kappa  \prod_{\substack { p \mid \mathrm g }, p \in  \mathcal P} {\frac{p -1} {p -2}}.   
\end{align*} 

б ОЮПЮЦПЮТЕ \ref{Eul_Gol_conj} ЛШ ОНЙЮФЕЛ, ВРН ВХЯКН ОПЕДЯРЮБКЕМХИ ВЕРМНЦН ВХЯКЮ $\mathrm g$ Б БХДЕ ЯСЛЛШ ДБСУ ОПНЯРШУ ЛНФЕР АШРЭ БШПЮФЕМН ВЕПЕГ 
${\pi}_2 (\mathrm g)$.
б ЩРНЛ ЯКСВЮЕ, Б ЯННРБЕРЯРБХЕ Я гЮЛЕВЮМХЕЛ \ref{rem_on_sum} Й ОЮПЮЦПЮТС \ref{reprofevennumber}, МСФМН ОНКНФХРЭ $\kappa = 0.5$ ЕЯКХ $3 \nmid {\mathrm g}$ Х $\kappa = 1$ ЕЯКХ $3 \mid {\mathrm g}$. 
  
щРХ ПЮЯЯСФДЕМХЪ ЛНФМН НАНАЫХРЭ МЮ ЯКСВЮИ ЯСЛЛШ КЧАНЦН ВХЯКЮ ОНЯКЕДНБЮРЕКЭМНЯРЕИ.
яСЛЛЮ $k$ РЮЙХУ, ЙЮЙ ${\mathcal A}$, ${\mathcal B}$, ${\mathcal A}^{m'}$ Х ${\mathcal B}^{m'}$ ОНЯКЕДНБЮРЕКЭМНЯРЕИ ЯНДЕПФХР МЕ ЛЕМЕЕ НДМНИ Х МЕ АНКЕЕ, ВЕЛ $k$
АЕЯЙНМЕВМШУ ОНДОНЯКЕДНБЮРЕКЭМНЯРЕИ МСКЕБШУ ВКЕМНБ, ХМДЕЙЯШ ЙНРНПШУ НАПЮГСЧР ЮПХТЛЕРХВЕЯЙХЕ ОПНЦПЕЯЯХХ Я ПЮГМНЯРЭЧ $p$ ДКЪ БЯЕУ $p \in \mathcal P$.
мЮОПХЛЕП, ЙНМЯРПСЙЖХЪ ${{\mathcal T}^{m'}} - {\mathcal T} = ({\mathcal R^{m'}} - {\mathcal L^{m'}}) - ({\mathcal R} - {\mathcal L})$ ДЮЕР ВХЯКН ТНПЛ, ЯНЯРНЪЫХУ ХГ ДБСУ ОЮП ОПНЯРШУ-АКХГМЕЖНБ Я ПЮГМНЯРЭЧ, ПЮБМНИ $6m'$. 
гДЕЯЭ ЛШ ХЛЕЕЛ $k=4$. 
оНДПНАМЕЕ ЩРЮ ЙНМЯРПСЙЖХЪ АСДЕР ПЮЯЯЛНРПЕМЮ Б О. \ref{patterns}.

\section{юДДХРХБМШЕ ОПНАКЕЛШ } \label{additiveproblem}

\subsection{цХОНРЕГЮ Н ОПНЯРШУ ВХЯКЮУ-АКХГМЕЖЮУ}\label{twinprimeconj}

\begin{proposition} \label{twinprime_conj}
яСЫЕЯРБСЕР ТСМЙЖХЪ $H(m)$ РЮЙЮЪ, ВРН: 
\begin{enumerate}
	\item  ${\pi}_2 (6m)>mH(m)$ ДКЪ БЯЕУ ДНЯРЮРНВМН АНКЭЬХУ $m$;
	\item $mH(m) \to \infty $ ЙНЦДЮ $m \to~\infty$.
\end{enumerate}
\end{proposition}
хГ МЕПЮБЕМЯРБ
$$ \pi (x) > \frac {x} {\log x} \left( 1 + \frac{1} {2 \log x} \right), \  x \geqslant 59; $$  
$$ \frac{{e^{ - \gamma } }} {{\log x}} \left( 1 + \frac{1} {2 (\log x)^2} \right) > \prod\limits_{p \leqslant x} {\left( {1 - \frac{1} {p}} \right)}, \ x > 1,$$

(Rosser, Schoenfeld \cite{Rosser-Schoenfeld}) 
ЯКЕДСЕР, ВРН ДКЪ БЯЕУ ДНЯРЮРНВМН АНКЭЬХУ $x$ БШОНКМЪЕРЯЪ МЕПЮБЕМЯРБН
\begin{align*} 
  \pi (x) > x {e^{ \gamma }} \prod\limits_{ p \leqslant x} {\left( {1 - \frac{1} {p}} \right)}.  
 \end{align*} 
оЕПЕУНДЪ Й ОЕПЕЛЕММШЛ $n$ Х $m$ ОНКСВХЛ 
\begin{align*} 
n {e^{ \gamma }} \prod\limits_{ p \leqslant n} {\left( {1 - \frac{1} {p}} \right)} <  \pi (a,n) + \pi (b,n), \ n \geqslant 31;  
 \end{align*} 
\begin{align} \label{single_sieve1}
  m {e^{ \gamma }} \prod\limits_{5 \leqslant p \leqslant 6m} {\left( {1 - \frac{1} {p}} \right)}  < \pi (a, 6m) ,\ m \geqslant 2; 
\end{align} %
\begin{align} \label{single_sieve2}
  m {e^{ \gamma }} \prod\limits_{5 \leqslant p \leqslant 6m} {\left( {1 - \frac{1} {p}} \right)}  < \pi (b, 6m),\ m \geqslant 9. 
\end{align} %
оНКНФХЛ ЯСЫЕЯРБНБЮМХЕ ТСМЙЖХИ $\varphi_{a,m}$ Х $\varphi_{b,m}$ РЮЙХУ, ВРН $m \varphi_{a,m}   = \pi (a,6m),m \varphi_{b,m}   =  \pi (b,6m).$
рНЦДЮ ХГ (\ref{single_sieve1}) Х (\ref{single_sieve2})  ЯКЕДСЕР, ВРН ДКЪ БЯЕУ ДНЯРЮРНВМН АНКЭЬХУ $m$ БШОНКМЪЧРЯЪ МЕПЮБЕМЯРБЮ
\begin{equation*} 
{e^{ \gamma }} \prod\limits_{5 \leqslant p \leqslant 6m} {\left( {1 - \frac{1} {p}} \right)} <  \frac{\pi (a,6m)}{m} = \varphi_{a,m} , \   
{e^{ \gamma }} \prod\limits_{5 \leqslant p \leqslant 6m} {\left( {1 - \frac{1} {p}} \right)} <  \frac{\pi (b,6m)}{m} =\varphi_{b,m}.  
\end{equation*}	
хГ РЕНПЕЛШ дХПХУКЕ Н ОПНЯРШУ ВХЯКЮУ Б ЮПХТЛЕРХВЕЯЙХУ ОПНЦПЕЯЯХЪУ ЯКЕДСЕР, ВРН 
$\varphi_{a,m} \sim  \varphi_{b,m} $ ЙНЦДЮ $m \to \infty$. 
дКЪ ОЕПЕУНДЮ Й ДБНИМНЛС ПЕЬЕРС ХЯОНКЭГСЕЛ НРМНЬЕМХЕ
\begin{align} \label{repited_twin}
{C_{(1 : 2)}(m)} &= {  \prod \limits_{5 \leqslant p \leqslant 6m}} \left( {1 - \frac{2}{p}} \right) \Big{/}
{  \prod \limits_{5 \leqslant p \leqslant 6m}} \left( {1 - \frac{1}{p}} \right)^2  \\
	&= { \prod \limits_{5 \leqslant p \leqslant 6m}} \frac{p(p-2)}{(p-1)^2} \nonumber \\
	&= \frac {4} {3} { \prod \limits_{3 \leqslant p \leqslant 6m}} \frac{p(p-2)}{(p-1)^2}. \nonumber 
	\end{align}
рНЦДЮ ХГ (\ref{repited_twin}), ХЛЕЪ Б БХДС МЕПЮБЕМЯРБa (\ref{single_sieve1}) Х (\ref{single_sieve2}), ОНКСВХЛ 
\begin{align}\label{H2}
H(m) &= { {e^{2 \gamma }} \prod \limits_{5 \leqslant p \leqslant 6m}} \left( {1 - \frac{2}{p}} \right)  \nonumber\\
     &= {  {C_{(1:2)} (m) }  {e^{2 \gamma }}  \prod \limits_{5 \leqslant p \leqslant 6m}} \left( {1 - \frac{1}{p}} \right)^2.
\end{align} 
мЕПЮБЕМЯРБН ${\pi}_2 (6m) > m H(m)$ БШОНКМЪЕРЯЪ ДКЪ БЯЕУ $m>5$ Х СДНБКЕРБНПЪЕР СЯКНБХЪЛ оПЕДКНФЕМХЪ \ref{twinprime_conj}.
дНЙЮГЮРЕКЭЯРБН БРНПНЦН СРБЕПФДЕМХЪ ЩРНЦН Х БЯЕУ ОНЯКЕДСЧЫХУ оПЕДКНФЕМХИ ОПХБНДХРЯЪ Б ЙНМЖЕ ЦКЮБШ.
тСМЙЖХЪ $mH(m)$ МЕ ЪБКЪЕРЯЪ УНПНЬХЛ ОПХАКХФЕМХЕЛ ВХЯКЮ ОЮП ОПНЯРШУ-АКХГМЕЖНБ: ${\pi}_2 (6m) - mH(m) > 1251$ СФЕ ОПХ $6m=10^6$ Б РН БПЕЛЪ, ЙЮЙ ${\pi}_2 (10^6)  = 8168$.

\subsubsection{оЕПБЮЪ ЦХОНРЕГЮ уЮПДХ-кХРРКБСДЮ.} 

еЯКХ ХГБЕЯРМШ ГМЮВЕМХЪ $\pi (a,6m), \pi (b,6m)$, РН ХГ (\ref{H2}) АСДЕЛ ХЛЕРЭ
\begin{align} \label{myfirstHLconj}
\pi'_2(6m) = C_{(1 : 2)}  (m) \cdot m  \frac {\pi (a,6m) } {m} \cdot \frac { \pi (b,6m)} {m}.   
\end{align}
оНДЯРЮМНБЙЮ $m=n/6$, ${\pi (a, 6m) \cdot \pi (b, 6m)} \sim (\pi(n) /2)^2$ Х $C_{(1:2)}  (m) = 4C_2 /3$, ЦДЕ $C_2 $ - ОНЯРНЪММЮЪ ОПНЯРШУ-АКХГМЕЖНБ, ДЮЕР 
\begin{align*} 
  \pi'_2(n) = 
	2 n C_2 \left[ \frac {\pi(n)} {n} \right]^2 
	\sim  2 C_2  \frac {n} {(\log n)^2}, 
	\end{align*}
Х ЯКЕДНБЮРЕКЭМН, (\ref{myfirstHLconj}) ЩЙБХБЮКЕМРМН ОЕПБНИ ЦХОНРЕГЕ уЮПДХ-кХРРКБСДЮ (ЯЛ. МЮОПХЛЕП \cite {Hardy-Littlwood_first_conj}). 
б ЩРНЛ ЯКСВЮЕ ${\pi}_2 (10^6) - \pi'_2(10^6) = 32.5356\ldots$ 
кСВЬЕЕ ОПХАКХФЕМХЕ ДКЪ ЙНКХВЕЯРБЮ ОЮП ОПНЯРШУ-АКХГМЕЖНБ БШПЮФЮЕРЯЪ ВЕПЕГ ОКНРМНЯРЭ ПЮЯОПЕДЕКЕМХЪ ОПНЯРШУ ВХЯЕК:
\begin{align} \label{firstHLconj}
  \pi_2(n) \sim 2 C_2 \int^n_2  \frac {dt} {(\log t)^2}.
\end{align}
сВХРШБЮЪ, ЙЮЙ ОНЙЮГЮМН Б ПЮГДЕКЕ \ref{primes_with_fixed_gap}, ВРН ОКНРМНЯРЭ ОПНЯРШУ ВХЯЕК МЮ НДМНЛ ХГ ЯСЛЛХПСЕЛШУ НРПЕГЙНБ ХГЛЕМЪЕРЯЪ НР $1/ \log 2$ ДН $1/ \log n$,
Ю МЮ ДПСЦНЛ НР $1/ \log (2 + \mathrm g)$ ДН $1/ \log (n + \mathrm g)$,
ДКЪ ВХЯКЮ ОЮП ОПНЯРШУ Я ПЮГМНЯРЭЧ ${\mathrm g}$ Х МЕ ОПЕБНЯУНДЪЫХУ $n + \mathrm g$, АСДЕЛ ХЛЕРЭ: 
\begin{align*} 
  \frac {{\pi}_{{\mathrm g}}(n)}{\eta_2 (\mathrm g)} 
		\sim 2 C_2 \int^n_2  \frac {dt} {\log t \log (t + \mathrm g) }.
\end{align*}

\subsection{цХОНРЕГЮ цНКЭДАЮУЮ - щИКЕПЮ}\label{Eul_Gol_conj}

\begin{proposition} 
яСЫЕЯРБСЕР ТСМЙЖХЪ $H'(m)$ РЮЙЮЪ, ВРН: 
\begin{enumerate}
	\item ДКЪ БЯЕУ ДНЯРЮРНВМН АНКЭЬХУ $m$ ВХЯКН ОПЕДЯРЮБКЕМХИ Б БХДЕ ЯСЛЛШ ДБСУ ОПНЯРШУ ЙЮФДНЦН ХГ РПЕУ ОНЯКЕДНБЮРЕКЭМШУ ВЕРМШУ ВХЯЕК $\mathrm g_m^1$, $\mathrm g_m^2$, $\mathrm g_m^3$ МЕ ЛЕМЭЬЕ $mH'(m)$;
	\item $mH'(m) \to \infty $ ЙНЦДЮ $m \to~\infty $.
\end{enumerate}
\end{proposition}%

бШПЮФЕМХЕ (\ref{firstHLconj}) НОПЕДЕКЪЕР ${\pi}_{\mathrm g} (n)$ ВЕПЕГ ОКНРМНЯРЭ ПЮЯОПЕДЕКЕМХЪ ОПНЯРШУ ВХЯЕК Б МЮРСПЮКЭМНЛ ПЪДЕ Я ХГЛЕМЕМХЕЛ ОКНРМНЯРХ Б ЙЮФДНЛ ХГ ЯКЮЦЮЕЛШУ НРПЕГЙНБ НР $1/ \log 2$ ДН $1/ \log n$.
оПХ ОНЯРПНЕМХЕ НРПЕГЙНБ ДКЪ НОПЕДЕКЕМХЪ ${\pi^+} ({\mathrm g})$ ЯЙКЮДШБЮЧРЯЪ НРПЕГЙХ, ОКНРМНЯРЭ ПЮЯОПЕДЕКЕМХЪ ОПНЯРШУ Б НДМНЛ ХГ ЙНРНПШУ СЛЕМЭЬЮЕРЯЪ, Ю Б ДПСЦНЛ ЯКЮЦЮЕЛНЛ НРПЕГЙЕ СБЕКХВХБЮЕРЯЪ. 
рЮЙХЛ НАПЮГНЛ, ЮМЮКНЦХВМН (\ref{firstHLconj}), ДКЪ НЖЕМЙХ ВХЯКЮ ОПЕДЯРЮБКЕМХИ ВЕРМНЦН ВХЯКЮ $\mathrm g$ Б БХДЕ ЯСЛЛШ ДБСУ ОПНЯРШУ ОНКСВХЛ: 
\begin{align*} 
  \frac {\pi^+ (\mathrm g)} {\eta_2 (\mathrm g)}  \sim 2 C_2 \int^{n-2}_2  \frac {dt} {\log t \log(n-t)},
\end{align*}
ЦДЕ ДКЪ ДНЯРЮРНВМН АНКЭЬХУ $\mathrm g$ ЛНФМН ОНКЮЦЮРЭ $n= \mathrm g$. гДЕЯЭ МСФМН ЯДЕКЮРЭ ЯННРБЕРЯРБСЧЫХИ БШАНП ГМЮВЕМХИ $\kappa$ Б ${\eta_2 (\mathrm g)}$. 
нРМНЬЕМХЕ
\begin{align*} 
  \mu_2 (n) &=   {\sum^{n-2}_2  \frac {1} {\log t \log(n-t)}} \Big{/}  {\sum^{n-2}_2  \frac {1} {( \log t)^2}} \\
  &\approx    {\int^{n-2}_2  \frac {dt} {\log t \log(n-t)}}  \Big{/} {\int^{n-2}_2  \frac {dt} {( \log t)^2}}, 
\end{align*}
ХЛЕЕР ЛХМХЛЮКЭМНЕ ГМЮВЕМХЕ $0,706\ldots$ ОПХ $n=32$, ДНЯРХЦЮЕР ГМЮВЕМХЪ $0,972\ldots$ ОПХ $n=10^5$. 
рЮЙХЛ НАПЮГНЛ, оПЕДКНФЕМХЧ 2 СДНБКЕРБНПЪЕР ТСМЙЖХЪ $H'(m) = \mu_2 (6m) H(m)$.
нЙНМВЮРЕКЭМН, ОПХ $\mathrm g \to~\infty $ АСДЕЛ ХЛЕРЭ 
$\pi^+(\mathrm g)  \sim \mu_2 (\mathrm g)  \eta_2 (\mathrm g)  {\pi}_2 {(\mathrm g)} $.
 
\subsection{сОНПЪДНВЕММШЕ МЮАНПШ ОПНЯРШУ ВХЯЕК (ЙНПРЕФХ) } \label{tuple}

пЮЯЯЛНРПХЛ ЙНПРЕФХ МЕВЕРМШУ ОПНЯРШУ ВХЯЕК, ЯНЯРНЪЫХЕ ХГ ВЕРШПЕУ ЩКЕЛЕМРНБ $q(g_1=0,g_2,g_3,g_4)$.
бНГЛНФМШ РПХ ЯКСВЮЪ: 
\begin{enumerate}
	\item $g_2,g_3,g_4 \equiv 0,2 \pmod 6 \Rightarrow g_1 \in A$;
	\item $g_2,g_3,g_4 \equiv 0,4 \pmod 6 \Rightarrow g_1 \in B$;
	\item $g_2 \equiv g_3 \equiv g_4 \equiv 0 \pmod 6 \Rightarrow g_1,g_2,g_3,g_4 \in A$ ХКХ $g_1,g_2,g_3,g_4 \in B$.
\end{enumerate}
дКЪ ОНЯРПНЕМХЪ ЙНМЯРПСЙЖХХ, ДЮЧЫЕИ ПЕЬЕМХЕ ГЮДЮВХ МЮУНФДЕМХЪ ЙНПРЕФЕИ $q(0,g_2,g_3,g_4)$, МЕНАУНДХЛН ЯКНФХРЭ ВЕРШПЕ ОНЯКЕДНБЮРЕКЭМНЯРХ БХДЮ ${\mathcal A}^{m'}$ 
Х ${\mathcal B}^{m'}$. 
йЮФДЮЪ ОНЯКЕДНБЮРЕКЭМНЯРЭ ЯННРБЕРЯРБСЕР ЙКЮЯЯС НДМНЦН ХГ ЩКЕЛЕМРНБ, ОПХ ЩРНЛ 

ЕЯКХ $g_1 \in A$, РН $m' = g/6$ ДКЪ $g \in A$ Х $m' = (g-2)/6$ ДКЪ $g \in B$; 

ЕЯКХ $g_1 \in B$, РН $m' = (g+2)/6$ ДКЪ $g \in A$ Х $m' = g/6$ ДКЪ $g \in B$.\\
мЮОПХЛЕП, ДКЪ $q(0,2,12,24)$ ОНКСВХЛ 
\tabcolsep=0.078em
\begin{center}
\begin{tabular}{rrrrrrrrrrrrrrrrrrrrrrrrrrr}
 ${\mathcal A} =$ &{5,}&{11,}&{17,}&{23,}&{29,}&{0,} &{41,}&{47,}&{53,}&{59,}&{0,} &{71,}&{0,}  &{83,}&{89,} &{0,}  &{101,}&{107,}&{113,}&{0,}  &{0,\ldots}\\ 
 ${\mathcal B} =$ & {7,}&{13,}&{19,}&{0,}&{31,}&{37,}&{43,}&{0,} &{0,} &{61,}&{67,}&{73,}&{79,} &{0,} &{0,}  &{97,} &{103,}&{109,}&{0,}  & {0,} &{127,\ldots} \\  
${\mathcal A}^2=$ &{17,}&{23,}&{29,}&{0,}&{41,}&{47,}&{53,}&{59,}&{0,} &{71,}&{0,} &{83,}&{89,} &{0,} &{101,}&{107,}&{113,}&{0,}  &{0,}  &{131,}&{137,\ldots} \\ 
 ${\mathcal A^4}=$ &{29,}&{0,}&{41,}&{47,}&{53,}&{59,}&{0,} &{71,}&{0,} &{83,}&{89,}&{0,} &{101,}&{107,}&{113,}&{0,} &{0,}  &{131,}&{137,}&{0,}  &{149,\ldots}\\ 
	 \hline 
$  S(0,2,12,24) =  $ &{1,} &{0,} &{1,}&{0,} &{1,} &{0,} &{0,} &{0,} &{0,} &{1,} &{0,} &{0,} &{0,}  &{0,} &{0,}  &{0,}  &{0,} &{0,} &{0,}   &{0,}  &{0,}\ldots
\end{tabular}
\end{center}

оНЯКЕДНБЮРЕКЭМНЯРЭ $S(0,2,12,24)$ ЕЯРЭ ПЕЮКХГЮЖХЪ ВЕРШПЕУЙПЮРМНЦН ПЕЬЕРЮ. 
дКЪ КЧАНЦН $p \in \mathcal{P}$ Б МЕИ МЮИДСРЯЪ ВЕРШПЕ ПЮГКХВМШЕ (МЕЯНБОЮДЮЧЫХЕ) ОНДОНЯКЕДНБЮРЕКЭМНЯРХ МСКЕБШУ ВКЕМНБ, ХМДЕЙЯШ ЙНРНПШУ НАПЮГСЧР ЮПХТЛЕРХВЕЯЙХЕ ОПНЦПЕЯЯХХ Я ПЮГМНЯРЭЧ $p$ . 
хЯЙКЧВЕМХЪЛХ ЪБКЪЧРЯЪ ОНДОНЯКЕДНБЮРЕКЭМНЯРХ Я ПЮГМНЯРЪЛХ $p=5$ Х $p=11$, РЮЙ ЙЮЙ $5 \mid g_3 - g_2$ Х $11 \mid g_4 - g_2$.
вХЯКН ЙЮФДНИ ХГ ЩРХУ ОНДОНЯКЕДНБЮРЕКЭМНЯРЕИ Б ОНЯКЕДНБЮРЕКЭМНЯРХ $S(0,2,12,24)$ ПЮБМН $3$.

\subsubsection{йНПРЕФХ, ЯНДЕПФЮЫХЕ ДБЕ ОЮПШ АКХГМЕЖНБ } \label{patterns}

оСЯРЭ ${\Pi}_{m'}(n)$ - ЙНКХВЕЯРБН ОЮП АКХГМЕЖНБ $p, p+2$ МЕ ОПЕБНЯУНДЪЫХУ $n$ Х РЮЙХУ, ВРН $p' = p +6m', p'+2 $ РЮЙФЕ ЪБКЪЕРЯЪ ОЮПНИ АКХГМЕЖНБ.
рЮЙЮЪ ВЕРБЕПЙЮ ЕЯРЭ ЙНПРЕФ $q(0,2,6m',2)$. 
йНМЯРПСЙЖХЕИ ДКЪ НОПЕДЕКЕМХЪ ВХЯКЮ РЮЙХУ ЙНПРЕФЕИ АСДЕР 
\begin{align} 
{\Pi}_{m'} (6m) &= \#  \{ {\mathcal T}^{m'}_{m}  - {\mathcal T}_{m} \} \nonumber \\
                &= \# \{  ({\mathcal R^{m'}} - {\mathcal L^{m'}}) - ({\mathcal R} - {\mathcal L}) \} \nonumber \\
                &=  \# \{  ({\mathcal B^{m'}} - {\mathcal A^{m'}}) - ({\mathcal B} - {\mathcal A}) \}.  \label{Q}
\end{align} 
оПХ $m' =1$ ОНКСВХЛ ВХЯКН ЙБЮДПСОКЕРНБ ${\Pi}_1 (6m)$. 

\begin{proposition} \label{patterns1} 
яСЫЕЯРБСЕР ТСМЙЖХЪ $Q(m)$ РЮЙЮЪ, ВРН: 
\begin{enumerate}
	\item ДКЪ БЯЕУ ДНЯРЮРНВМН АНКЭЬХУ $m$ БШОНКМЪЕРЯЪ МЕПЮБЕМЯРБН ${\Pi}_1 (6m)> mQ(m)$;
	\item $mQ(m) \to \infty $ ЙНЦДЮ $m \to \infty $.
\end{enumerate}
\end{proposition}
йЮЙ АШКН ГЮЛЕВЕМН Б ПЮГДЕКЕ \ref{general_sieving}, Б ЙНМЯРПСЙЖХХ (\ref{Q}) ЯСЛЛХПСЧРЯЪ $k=4$ НРПЕГЙЮ УНПНЬН ЯРПСЙРСПХПНБЮММШУ ОНЯКЕДНБЮРЕКЭМНЯРЕИ, ОПНЯЕЪММШУ ОН БЯЕЛ $p \in \mathcal P$. 
оПХ $m' =1$ ДКЪ БЯЕУ $p \in \mathcal P$ 
\begin{align*}
	\left|   \left\{  i \colon i \leqslant m, (p, a_i \cdot b_i \cdot a_{i+1} \cdot b_{i+1} ) =1  \right\}  \right|  \sim m \left( 1 - 4/p \right), 
		\end{align*}  
РН ЕЯРЭ ПЕГСКЭРЮРНЛ ЯКНФЕМХЪ ВЕРШПЕУ ОНЯКЕДНБЮРЕКЭМНЯРЕИ $ \mathcal A, \mathcal B, \mathcal A^1$ Х $\mathcal B^1$ ЪБКЪЕРЯЪ ОНЯКЕДНБЮРЕКЭМНЯРЭ, ЦДЕ ДКЪ ЙЮФДНЦН 
$p \in \mathcal P$ ХЛЕЕРЯЪ ПНБМН ВЕРШПЕ ПЮГКХВМШЕ (МЕЯНБОЮДЮЧЫХЕ) ОНДОНЯКЕДНБЮРЕКЭМНЯРХ МСКЕБШУ ВКЕМНБ, ХМДЕЙЯШ ЙНРНПШУ НАПЮГСЧР ЮПХТЛЕРХВЕЯЙХЕ ОПНЦПЕЯЯХХ Я 
ПЮГМНЯРЭЧ $p$.  
хЯОНКЭГСЪ ЛЕРНД, ЮМЮКНЦХВМШИ ОПХЛЕМЕММНЛС Б ПЮГДЕКЕ \textbf{\ref{twinprimeconj}}, ДКЪ ОЕПЕУНДЮ НР НДМНЙПЮРМНЦН Х ДБНИМНЦН ПЕЬЕРЮ Й ВЕРШПЕУЙПЮРМНЛС ПЕЬЕРС ЯННРБЕРЯРБЕММН ХЯОНКЭГСЕЛ НРМНЬЕМХЪ  
 \begin{align} 
  C_{(1:4)} (m) &={\prod \limits_{5 \leqslant p \leqslant 6m}} \left( 1- \frac{4}{p}\right) \Big{/} 
	             {\prod \limits_{5 \leqslant p \leqslant 6m}} \left( 1- \frac{1}{p} \right)^4 \label{Q_41}\\
	&= { \prod \limits_{5 \leqslant p \leqslant 6m}}  \frac{(p-4)p^3}{(p-1)^4}, \nonumber\\
  C_{(2:4)} (m) &={\prod \limits_{5 \leqslant p \leqslant 6m}} \left( 1- \frac{4}{p}\right) \Big{/}
	           {\prod \limits_{5 \leqslant p \leqslant 6m}} \left[ \left( 1- \frac{2}{p} \right) ^2 \right]^2 \label{Q_42}\\
						&= { \prod \limits_{5 \leqslant p \leqslant 6m}}  \frac{(p-4)p}{(p-2)^2}.\nonumber
\end{align} %
рНЦДЮ ХГ (\ref{Q_41}) Х (\ref{Q_42}), ХЛЕЪ Б БХДС МЕПЮБЕМЯРБa (\ref{single_sieve1}), (\ref{single_sieve2}) Х (\ref{H2}), ОНКСВХЛ
\begin{align*} 
  Q(m) &= { {e^{4 \gamma }} \prod \limits_{5 \leqslant p \leqslant 6m}} \left( {1 - \frac{4}{p}} \right) \\ 
	     &= C_{(1:4)} (m) {e^{4 \gamma }}{\prod \limits_{5 \leqslant p \leqslant 6m}} \left( 1- \frac{1}{p} \right)^4 \\
	     &= C_{(2:4)} (m) {e^{4 \gamma }}{\prod \limits_{5 \leqslant p \leqslant 6m}} \left[ \left( 1- \frac{2}{p} \right) ^2 \right]^2
\end{align*}  
мЕПЮБЕМЯРБН ${\Pi}_1 (6m)> mQ(m)$ БШОНКМЪЕРЯЪ ДКЪ БЯЕУ $m>1$. оПХ $6m=10^6$ ХЛЕЕЛ ${\Pi}_1 (6m)= 166$ Х ${\Pi}_1 (6m)- mQ(m) \approx 52.07$.
кСВЬСЧ НЖЕМЙС ДКЪ ${\Pi}_1 (6m)$ ЛНФМН ОНКСВХРЭ ВЕПЕГ ВХЯКН ОПНЯРШУ ВХЯЕК ХКХ ВХЯКН ОЮП ОПНЯРШУ-АКХГМЕЖНБ 
ХЯОНКЭГСЪ РЕУМХЙС, ОПХЛЕМЕММСЧ Б О. \textbf{\ref{twinprimeconj}}: 
\begin{align} 
   \Pi_1(6m) &\sim  m  C_{(1:4)} (m)\left[ \frac {\pi (a, 6m) } {m} \cdot \frac { \pi (b, 6m)} {m} \right]^2 \label{qoud_prime}  \\
    &\sim  m {C_{(2:4)} (m)} \left[ \frac {\pi_2(6m)} {m} \right]^2. \label{qoud_twin}
\end{align}
ЦДЕ $ \lim_{m \to + \infty} C_{(1:4)} (m) = 0,3074950$.
оПХ $n=10^6$,  ТНПЛСКЮ (\ref{qoud_prime}) ДЮЕР НРЙКНМЕМХЕ НР ХЯРХММНЦН ГМЮВЕМХЪ ВХЯКЮ ЙБЮДПСОКЕРНБ, ПЮБМНЕ $8.39\ldots$, Ю ТНПЛСКЮ (\ref{qoud_twin})  --- $7.12\ldots$

б НАЫЕЛ ЯКСВЮЕ 
${\Pi}_{m'} (6m) > \eta_4 (0,2,6m',2) {\Pi}_{1} (6m)$, ЦДЕ 
\begin{align*}  
    \eta_4 (0,2,6m',2)  
		= 	\frac {{\Pi}_{m'} (6m) } {{\Pi}_{1} (6m)} 
		=  \prod_{\substack { p \mid m'}, p \in  \mathcal P}   \left(\frac{p - 2} {p-4} \right) 
		\end{align*} 
вХЯКН ЙБЮДПСОКЕРНБ ${\Pi}_{1}(n)$, БШПЮФЕММНЕ ВЕПЕГ ОКНРМНЯРЭ ПЮЯОПЕДЕКЕМХЪ ОПНЯРШУ ВХЯЕК ОПХ $n \to \infty$ ЕЯРЭ
\begin{align*} 
  {\Pi}_{1}(n) &\sim  C_{(1:4)} (m) \int^n_2  \frac {dt} {\log t \log (t+2) \log (t+6) \log (t+8)} \\
	             &\sim  C_{(1:4)} (m) \int^n_2  \frac {dt} {(\log t)^4}.
\end{align*}
нАНАЫЕМХЕЛ оПЕДКНФЕМХЪ \ref{patterns1} ЪБКЪЕРЯЪ СРБЕПФДЕМХЕ Н ЯСЫЕЯРБНБЮМХХ АЕЯЙНМЕВМНЦН ЛМНФЕЯРБЮ КЧАШУ ЙНПРЕФЕИ, ЯНЯРНЪЫХУ ХГ ВЕРШПЕУ, МЕНАЪГЮРЕКЭМН ОНЯКЕДНБЮРЕКЭМШУ, ОПНЯРШУ ВХЯЕК. 
гЮЛЕРХЛ, ВРН $\eta_4 (g_1,g_2,g_3,g_4) \geqslant 1$ ДКЪ КЧАШУ ЙНПРЕФЕИ $q(g_1,g_2,g_3,g_4) $.

\subsection{сЯХКЕММЮЪ ЦХОНРЕГЮ цНКЭДАЮУЮ-щИКЕПЮ } \label{extGEconj}

оПХ ПЮЯЯЛНРПЕМХХ ЮДДХРХБМШУ ГЮДЮВ ЕЯРЕЯРБЕММШЛ НАПЮГНЛ БНГМХЙЮЕР БНОПНЯ: ЙЮЙ ЛМНЦН (ЙНМЕВМН КХ ВХЯКН) ВЕРМШУ ВХЯЕК, МЕ ОПЕДЯРЮБХЛШУ Б БХДЕ ЯСЛЛШ ДБСУ ОПНЯРШУ, ЙЮФДШИ ХГ ЙНРНПШУ ХЛЕЕР ЯБНЕЦН АКХГМЕЖЮ (Б ДЮКЭМЕИЬЕЛ - ДБСУ АКХГМЕЖНБ). нВЕБХДМН, ВРН ЕЯКХ НДМН ВХЯКН ХГ РПНИЙХ ВЕРМШУ ВХЯЕК $\mathrm g_m^1$, $\mathrm g_m^2$, $\mathrm g_m^3$ ОПЕДЯРЮБХЛН Б БХДЕ ЯСЛЛШ ДБСУ АКХГМЕЖНБ, РН ОПЕДЯРЮБХЛШ  Б РЮЙНЛ БХДЕ Х ДБЮ ДПСЦХУ. 
дБЕ ОЮПШ АКХГМЕЖНБ ДЮЧР ОН НДМНЛС ОПЕДЯРЮБКЕМХЧ ВХЯЕК  $\mathrm g_m^1$ Х $\mathrm g_m^3$ Х ДБЮ ОПЕДЯРЮБКЕМХЪ ВХЯКЮ $\mathrm g_m^2$ ДКЪ МЕЙНРНПНЦН $m$.
лШ АСДЕЛ ЯВХРЮРЭ РЮЙСЧ ЙНЛАХМЮЖХЧ ГЮ НДМН ОПЕДЯРЮБКЕМХЕ. рН ЕЯРЭ БЯЕ РПХ ВХЯКЮ ХКХ ОПЕДЯРЮБХЛШ Б БХДЕ ЯСЛЛШ ДБСУ АКХГМЕЖНБ, ХКХ МЕ ОПЕДЯРЮБХЛШ. 
мЕОНЯПЕДЯРБЕММЮЪ ОПНБЕПЙЮ АНКЕЕ ВЕЛ РПХДЖЮРХ РШЯЪВ РЮЙХУ РПНЕЙ ОНЙЮГЮКЮ, ВРН РНКЭЙН $12$ ХГ МХУ Б РЮЙНЛ БХДЕ МЕ ОПЕДЯРЮБХЛШ. 
\begin{proposition}{\textnormal{(сЯХКЕММЮЪ ЦХОНРЕГЮ цНКЭДАЮУЮ-щИКЕПЮ)} } \label{extGEconj1}
{бЯЕ ВЕРМШЕ ВХЯКЮ ЙПНЛЕ ВХЯЕК БХДЮ $\mathrm g_m^1$, $\mathrm g_m^2$, $\mathrm g_m^3$ ОПХ $m= 1,16,67, 86, 131, 151, 186, $ $ 191, 211, 226, 541, 701$, ХЛЕЧР ОПЕДЯРЮБКЕМХЪ Б БХДЕ ЯСЛЛШ ДБСУ АКХГМЕЖНБ. йПНЛЕ РНЦН, ЯСЫЕЯРБСЕР ТСМЙЖХЪ $Q'(m)$ РЮЙЮЪ, ВРН: 
\begin{enumerate}
	\item ДКЪ БЯЕУ ДНЯРЮРНВМН АНКЭЬХУ $m$ ЙНКХВЕЯРБН ЩРХУ ОПЕДЯРЮБКЕМХИ МЕ ЛЕМЭЬЕ $mQ'(m)$;
	\item $mQ'(m) \to \infty $ ЙНЦДЮ $m \to~\infty $.
\end{enumerate}}
\end{proposition}%
йНМЯРПСЙЖХЧ ДКЪ НОПЕДЕКЕМХЪ ЙНКХВЕЯРБЮ РЮЙХУ ОПЕДЯРЮБКЕМХИ ЛНФМН БШПЮГХРЭ БН ББЕДЕММШУ НАНГМЮВЕМХЪУ ЙЮЙ 
	  $\left( {{\mathcal T}_{m} + {\mathcal T}'_{m} } \right) = 
	  \Bigl(  \left({ {\mathcal R}_{m}  - {\mathcal L}_{m} } \right)  +   \left( { {\mathcal R}'_{m}  - {\mathcal L}'_{m} }  \right) \Bigr)$.
оПНДНКФЮЪ ПЮЯЯСФДЕМХЪ О.{\ref{Eul_Gol_conj}}, ОНКСВХЛ НЖЕМЙС ДКЪ $mQ'(m)$: 
\begin{align*}  \label{firsHLtconjforquad2}
 Q'(m)  <  {\eta_4 (m) \cdot \mu_4 (m) \cdot C_{(1:4)} (m)}  \left[ \frac {\pi(6m)} {m} \right]^4,
 \end{align*} 
ЦДЕ
\begin{align*} 
  \mu_4 (n) &=   {\sum^{n-2}_2  \frac {1} {(\log t)^2 (\log(n-t))^2}} \Big{/}   {\sum^{n-2}_2  \frac {1} {(\log t)^4}} \\
	&\approx   {\int^{n-2}_2  \frac {dt} {(\log t)^2 (\log(n-t))^2}}  \Big{/}   {\int^{n-2}_2  \frac {dt} {(\log t)^4 }} 
\end{align*}
ХЛЕЕР ЛХМХЛЮКЭМНЕ ГМЮВЕМХЕ $0,136278\ldots$ ОПХ $n=227$ Х  
ДНЯРХЦЮЕР ГМЮВЕМХЪ $0,64123\ldots$ ОПХ $n=2 \cdot 10^5$. рЮЙХЛ НАПЮГНЛ, ЛНФМН ОПХМЪРЭ $Q'(m) = \mu_4 (6m) Q(m) $. 
мЕПЮБЕМЯРБН $Q'(m) >1$ БШОНКМЪЕРЯЪ ДКЪ БЯЕУ $m \geqslant 947 \geqslant 701$ НРЙСДЮ ЯКЕДСЕР, ВРН БЯЕ ВЕРМШЕ ВХЯКЮ, АНКЭЬХЕ ВЕЛ $4208 = 6 \times 701 + 2$, ОПЕДЯРЮБХЛШ Б БХДЕ ЯСЛЛШ ДБСУ АКХГМЕЖНБ. \\
оПЕДКНФЕМХЕ \ref{extGEconj1} ЛНФЕР АШРЭ НАНАЫЕМН МЮ ЛМНФЕЯРБН ОЮП ОПНЯРШУ Я ТХЙЯХПНБЮММНИ ПЮГМНЯРЭЧ $d$ ЕЯКХ $3 \nmid d$.

яОПЮБЕДКХБНЯРЭ БРНПНЦН ОСМЙРЮ ЙЮФДНЦН ХГ оПЕДКНФЕМХИ ЯКЕДСЕР ХГ ОПХБЕДЕММНИ МХФЕ рЕНПЕЛШ.
\begin{theorem}\label{main_theoremj}
еЯКХ $m \to~\infty $, РН
\begin{align*} 
   { m  \prod \limits_{5 \leqslant p \leqslant 6m}} \left( {1 - \frac{4}{p}} \right) \to~\infty.
\end{align*} 
\end{theorem}
\begin{proof} 
хГ ЮЯХЛОРНРХВЕЯЙНЦН ПЮБЕМЯРБЮ (Rosser, Schoenfeld, \cite{Rosser-Schoenfeld})
\begin{align*}
   \prod \limits_{\alpha < p \leqslant x} \left( 1 - \frac {\alpha} {p} \right) 
   \cong \frac {c(\alpha)} {(\log x)^\alpha}
\end{align*}
ЛШ ХЛЕЕЛ 
\begin{align*}
 I_m = m \prod \limits_{5 \leqslant p \leqslant 6m} \left( 1 - \frac {4} {p} \right) 
   \cong    \frac { c(4) \cdot   m } { (\log 6m)^4},
\end{align*}
ЦДЕ $c(4) > 2.47 $.
дНЙЮГЮРЕКЭЯРБН ЯКЕДСЕР ХГ ХГБЕЯРМНЦН НРМНЬЕМХЪ $ \lim_{x \to + \infty} x^{-n} \log{x} = 0$, ЯОПЮБЕДКХБНЦН ДКЪ БЯЕУ $n>0$.
(мЮОПХЛЕП, Hardy \cite {MainInneq}, ЯРП. 381.) 

\end{proof}

\section{оКНРМНЯРХ ОНЯКЕДНБЮРЕКЭМНЯРЕИ} 

оСЯРЭ ЙЮЙ НАШВМН, МЮРСПЮКЭМЮЪ (ЮЯХЛОРНРХВЕЯЙЮЪ) ОКНРМНЯРЭ $\delta \left( {S} \right)$ Х ОКНРМНЯРЭ ОН ьМХПЕКЭЛЮМС $d \left( {S} \right)$ ОНЯКЕДНБЮРЕКЭМНЯРХ  $S_m$ НОПЕДЕКЪЧРЯЪ ЙЮЙ
\begin{align*}
 \delta \left( {S} \right)  =  \lim_{m \to \infty}  \frac { \left|S_{(m)}\right|} {m}, \, 
        d \left( {S} \right)  =  \liminf_{m \to \infty}  \frac { \left| S_{(m)} \right|} {m}. 
\end{align*}
нОПЕДЕКХЛ ЯСЛЛС $S' \oplus S''$ ДБСУ ОНЯКЕДНБЮРЕКЭМНЯРЕИ ЙЮЙ ЛМНФЕЯРБН ЖЕКШУ ВХЯЕК БХДЮ $s'$, $s''$ ХКХ $s' + s''$.
нВЕБХДМН, ВРН $\delta ( P ) = d ( P ) =0$, 
$\delta \left( \mathcal L \right) = \delta \left( \mathcal R \right) = \delta \left( \mathcal T \right)  = 0$ Х
$d \left( \mathcal {L} \right) = d \left( \mathcal {R} \right) = d \left( \mathcal {T} \right) = 0$. 
кЕЦЙН БХДЕРЭ, ВРН ХГ ЯОПЮБЕДКХБНЯРХ ЦХОНРЕГШ цНКЭДАЮУЮ-щИКЕПЮ ЯКЕДСЕР, ВРН
$\delta \left( \emph {P} \oplus \emph {P} \right) > 0.5$.    
лЕРНД, ОПХЛЕМЕММШИ ДКЪ ДНЙЮГЮРЕКЭЯРБЮ ЩРНИ ЦХОНРЕГШ ОНГБНКЪЕР СРБЕПФДЮРЭ, ВРН
$$ d \left( {{\mathcal L} \oplus {\mathcal L}} \right) = 
d \left( {{\mathcal L} \oplus {\mathcal R}} \right) = d \left( {{\mathcal R} \oplus {\mathcal R}} \right)=1, $$
РН ЕЯРЭ ЯСЫЕЯРБСЧР МЮРСПЮКЭМШЕ АЮГХЯШ ОНПЪДЙЮ $2$ МСКЕБНИ ОКНРМНЯРХ (гЮЛЕВЮМХЕ \ref{rem_on_density}). 
хГ ЯОПЮБЕДКХБНЯРХ СЯХКЕММНИ ЦХОНРЕГШ цНКЭДАЮУЮ-щИКЕПЮ РЮЙФЕ ЯКЕДСЕР, ВРН
$ \delta \left( {{\mathcal T} \oplus {\mathcal T}} \right) = 1 $,
РН ЕЯРЭ ОНЯКЕДНБЮРЕКЭМНЯРЭ ${\mathcal T}$ ЪБКЪЕРЯЪ ЮЯХЛОРНРХВЕЯЙХЛ АЮГХЯНЛ ОНПЪДЙЮ $2$. 

оНЯКЕДНБЮРЕКЭМНЯРЭ ОПНЯРШУ ВХЯЕК (Я БЙКЧВЕМХЕЛ $1$) ЪБКЪЕРЯЪ АЮГХЯНЛ ОНПЪДЙЮ $3$. 

оНЯКЕДНБЮРЕКЭМНЯРЭ ОПНЯРШУ-АКХГМЕЖНБ ЪБКЪЕРЯЪ ЮЯХЛОРНРХВЕЯЙХЛ АЮГХЯНЛ ОНПЪДЙЮ $3$.

\end{document}